\documentclass[12pt]{article}
\usepackage{amsxtra,amssymb,amsmath,graphicx,bbm,amsthm,color,mathrsfs,enumerate}
\usepackage[urlcolor=black,linkcolor=black,citecolor=black,colorlinks=true]{hyperref}
\usepackage[active]{srcltx}
\usepackage{t1enc}
\usepackage{aurical}
\usepackage[T1]{fontenc}
\usepackage{comment}
\usepackage[left=2.5cm, right=2.5cm, top=2cm, bottom=4cm]{geometry}
\usepackage{caption}  

\newcommand{\e}{\mathbb{E}}
\newcommand{\E}{\mathbb{E}}

\renewcommand{\P}{\mathbb{P}}
\newcommand{\R}{\mathbb{R}}

\newcommand{\1}{{\mathbf 1}}

\def\be{\begin{align}}
\def\ee{\end{align}}
\def\b*{\begin{eqnarray*}}
\def\e*{\end{eqnarray*}}

\def\vp{\varphi}

\def\be{\begin{eqnarray}}
\def\ee{\end{eqnarray}}
\def\beq{\begin{equation}}
\def\eeq{\end{equation}}
\def\b*{\begin{eqnarray*}}
\def\e*{\end{eqnarray*}}
\def\bi{\begin{itemize}}
\def\ei{\end{itemize}}

\def \1{{\bf 1}}
\def\vp{\varphi}
\def\eps{\varepsilon}

\def\={\;=\;}

\def\x{\times}

\def \proof{{\noindent \bf Proof. }}
\def \ep{\hbox{ }\hfill$\Box$}
 \def\reff#1{{\rm(\ref{#1})}}

 \def\vs#1{\vspace{#1mm}}

\def\red#1{\textcolor[rgb]{1,0,0}{ #1}}

\def \E{\mathbb{E}}
\def \F{\mathbb{F}}

\def \N{\mathbb{N}}
\def \P{\mathbb{P}}

\def \R{\mathbb{R}}
\def\T{\mathbb{T}}

\def\Bc{{\cal B}}

\def\Ec{{\cal E}}
\def\Fc{{\cal F}}

\def\Lc{{\cal L}}

\def\Nc{{\cal N}}
\def\Kc{{\cal K}}

\def\Lc{{\cal L}}

\def\Xb{\bar X}

\def\Zb{\bar Z}

\newtheorem{Theorem}{Theorem}[part]
\newtheorem{Definition}{Definition}[part]
\newtheorem{Proposition}{Proposition}[part]
\newtheorem{Assumption}{Assumption}[part]
\newtheorem{Lemma}{Lemma}[part]
\newtheorem{Corollary}{Corollary}[part]
\newtheorem{Remark}{Remark}[part]

\makeatletter \@addtoreset{equation}{section}

\@addtoreset{Definition}{section}

\@addtoreset{Theorem}{section}

\@addtoreset{Proposition}{section}

\@addtoreset{Assumption}{section}

\@addtoreset{Corollary}{section}

\@addtoreset{Lemma}{section}

\@addtoreset{Remark}{section}

\@addtoreset{Example}{section}

\def\qr{{\rm q}}
\def\Qr{{\rm Q}}
\def\kr{{\rm k}}
\def\Xb{\mathbf{X}}
\def\Mb{\mathbf{M}}
\def\vr{{\rm v}}
\def\Kc{{\mathcal K}}
\def\Zb{\mathbf Z}
\def\as{\rm a.s.}
\def\T{{\rm T}}

\def\zr{{\rm z}}

\def\U{{\rm U}}
\def\Ab{\mathbf A}
\def\Ec{{\rm E}}

\begin{document}

\title{Optimal control under uncertainty and Bayesian parameters adjustments  }

\author{N. Baradel\thanks{ENSAE-ParisTech, CREST, and, Université Paris-Dauphine, PSL Research University, CNRS, UMR [7534], CEREMADE, 75016 Paris, France. }, B. Bouchard\thanks{Université Paris-Dauphine, PSL Research University, CNRS, UMR [7534], CEREMADE, 75016 Paris, France. This research is supported by the Initiative de Recherche  ``Stratégies de Trading et d'Investissement Quantitatif'',   Kepler-Chevreux and Collège de France. B. Bouchard is supported in part by the ANR project CAESARS (ANR-15-CE05-0024).}, N.~M. Dang\thanks{JVN Institute, VNU HCM. This research is funded by Vietnam National University HoChiMinh City (VNU-HCM) under grant number C2015-42-03.}}
\date{First version: January 2017, Revised version: September 2017}
\maketitle

\begin{abstract} We propose a general framework for studying optimal impulse control problem in the presence of uncertainty on the parameters.   Given a prior on the distribution of the unknown parameters, we explain how it should evolve according to the classical Bayesian rule  after each impulse. Taking these progressive prior-adjustments into account, we characterize the optimal policy through a quasi-variational parabolic equation, which can be solved numerically. The derivation of the dynamic programming equation seems to be new in this context. The main difficulty lies in the nature of the set of controls which depends in a non trivial way on the initial data through the filtration itself. 
  \end{abstract}
 
\vs2


 \section{Introduction}

We consider a general optimal impulse control problem under parameter uncertainty. This work is motivated by optimal trading problems. In this domain,  several market parameters are of major importance. It can be the nature of the market impact of aggressive orders, or the time to be executed when entering a book order queue, see e.g.~\cite{lehalle2013market} and the references therein.  However, the knowledge of these execution conditions is in general not perfect. One can try to estimate them but they remain random and can change from one market/platform to another one, or depending on the current market conditions. Most importantly, they can only be estimated  by actually acting on the market. We therefore face the typical problem of estimating a reaction parameter  (impact/execution time) while actually controlling a system (trading) that depends on these parameters. 

Such problems have been widely studied in the discrete time stochastic optimal control literature, see e.g.~\red{\cite{easley,HL.12}} for references. One fixes a certain prior distribution on the unknown parameter, and re-evaluate it each time an action is taken, by applying the standard Bayesian rule to the observed reactions. The optimal strategy   generically results from a compromise between acting on the system, to get more information, and being not too  aggressive, because of the uncertainty on the real value of the parameters.   
 If the support of the initial prior contains the true value of the parameters, one can expect (under natural identification conditions) that the sequence of updated priors actually converges to it in the long range. 

  It is a-priori much more difficult to handle in a continuous time framework with continuous time monitoring, as it leads to a filtering problem, leaving on an infinite dimensional space. However,  optimal trading can very naturally be considered in the impulse form, as orders are sent in a discrete time manner. In a sense, we are back to a discrete time problem whose dimension can be finite (depending on the nature of the uncertainty), although interventions on the system may occur at any time.  

In this paper, we thus  consider  a general impulse control problem with an unknown parameter, under which an initial prior law is set. Given this prior, we aim at maximizing a certain gain functional. We   show that the corresponding value function can be characterized as the unique viscosity solution (in a suitable class) of a quasi-variational parabolic equation. We also allow for (possibly) not observing immediately the effect of an impulse. This applies to any situations in which the effect of an impulse is observed only with delay, e.g.~nothing is observed but the execution time when an order is sent to a dark pool.  

The study of such non-classical impulse control problems seems to be new in the literature. From the mathematical point of view, the main difficulty consists in establishing a dynamic programming principle. The principal reason lies in the choice of the filtration. Because of the uncertainty on the parameter driving the dynamics, the only natural filtration to which the control policy should be adapted is the one generated by the controlled process himself. This implies in particular that the set of admissible controls depends heavily (and in a very non trivial way) on the initial state of the system at the starting time of the strategy. Hence, no a priori regularity nor good measurability properties can be expected to construct explicitly measurable almost optimal controls, see e.g.~\cite{bouchard2011weak}, or to apply a measurable selection theorem, see e.g.~\cite{dimitri1996stochastic}.   We therefore proceed differently. The (usually considered as) easy part of the dynamic programming can actually be proved, as it only requires a conditioning argument. It leads as usual to a sub-solution characterization. We surround the difficulty in proving the second (difficult) part by considering a discrete time version of our initial continuous time control problem. When the time step goes to $0$, it provides a super-solution of the targeted dynamic programming equation. Using comparison and the natural ordering on the value functions associated to the continuous and the discrete time model, we show that the two coincide at the limit.

Applications to optimal trading and an example of numerical scheme are provided in {the application paper } \cite{BBD16extended}.

The rest of the paper is organized as follows. The model is described in Section 2. In Section 3, we provide the PDE characterization of the value function. Proofs are collected in Section \ref{sec: proof disco}.  A sufficient condition for comparison to hold is provided in Section \ref{preuve_comparaison}. 
 
\parindent=0pt

\section{The impulse problem with parameters adjustment}\label{sec: setup}

All over this paper,   $C([0,T],\R^{d}) $ is the space of continuous functions from $[0, T]$ into $\mathbb{R}^{d}$ which start at $0$ at the origin. Recall that it is a Polish space for the sup-norm topology. We denote by $W(\omega)=\omega$ the canonical process and let $\P$ be the Wiener measure. 
We also consider a Polish space $(\U,\Bc(\U))$  that will support an unknown parameter $\upsilon$. We denote by $\Mb$ a locally compact subset\footnote{{In many situations, the family of probability measures of interest will in fact be parameterized or  be the set of measures on a compact metrizable space, see Remark \ref{rem: loc compact} below.} } of the set of Borel probability measures on $\U$ endowed with the topology of weak convergence. In particular,  it is  Polish.  A prior on the unknown parameter $\upsilon$ will be an element $m \in \Mb$. 
To allow for additional randomness in the measurement of the effects of actions on the system, we consider another Polish space $\Ec$ on which is defined a family $(\epsilon_{i})_{i \geq 0}$ of i.i.d.~random variables with common measure $\mathbb{P}_{\epsilon}$ on $\Ec$. On the product space $ {\Omega} := C([0,T],\R^{d}) \times \U \times \Ec^{\mathbb{N}}$, we consider the family of measures $\{\mathbb{P} \times m \times \mathbb{P}_{\epsilon}^{\otimes \mathbb{N}} : m \in \Mb\}$ and denote by $\mathbb{P}_{m}$ an element of this family whenever $m \in \Mb$ is fixed. The operator $\mathbb{E}_{m}$ is the expectation associated to $\mathbb{P}_{m}$. Note that $W$, $\upsilon$ and $(\epsilon_{i})_{i\ge 0}$ are independent under each $\P_{m}$. 
For $m\in \Mb$ given, we let $\F^{m}=(\Fc^{m}_{t})_{t\ge 0}$ denote the $\P_{m}$-augmentation of the filtration $ \F= (  \Fc_{t})_{t\ge 0}$ defined by $  \Fc_{t}=\sigma((W_{s})_{s\le t}, \upsilon,(\epsilon_{i})_{i \geq 0})$ for $t\ge 0$. Hereafter, all the random variables are considered with respect to the probability space $(\Omega,\Fc_{T}^{m})$ with $m\in \Mb$ given by the context, and where $T$ is a fixed time horizon.

\subsection{The controlled system}

Let  $\Ab \subset [0, T] \times \mathbb{R}^{d}$ be  a (non-empty) compact set.  Given $N\in \N$ and $m\in \Mb$, we denote by  $\Phi^{\circ,m}_{N}$ the collection of sequences of random variables $\phi=(\tau_{i},\alpha_{i})_{i\ge 1}$ on $(\Omega,\Fc_{T}^{m})$ with values in $\R_{+}\x \Ab$ such that $(\tau_{i})_{i\ge 1}$ is a non-decreasing sequence of $\F^{m}$-stopping times satisfying $\tau_{j}>T$ $\P_{m}-\as$~for $j>N$.  We set 
    \[
        \Phi^{\circ,m} := \bigcup_{N \geq 1}\Phi^{\circ,m}_{N}.
    \]
 An element $\phi=(\tau_{i},\alpha_{i})_{1\le i\le N}\in \Phi^{\circ,m}$ will be our impulse control and we write $\alpha_{i}$ in the form 
 $$
 \alpha_{i}=(\ell_{i},\beta_{i}) \mbox{ with }  \ell_{i}\in [0,T] \mbox{ and }\beta_{i}\in \R^{d} \;\P_{m}-\as
 $$  
 More precisely, the $\tau_{i}$'s will be the times at which an impulse is made on the system (e.g.~a trading robot is launched), $\beta_{i}$ will model the nature of the order send at time $\tau_{i}$ (e.g.~the parameters used for the trading robot), and $\ell_{i}$ will stand for the maximal time length during which  no new intervention on the system can be made (e.g.~the time prescribed to the robot to send orders on the market). Later on we shall impose more precise non-anticipativity conditions.  
\vs2

From now on, we shall always use the notation $(\tau^{\phi}_{i},\alpha^{\phi}_{i})_{i\ge 1}$ with $\alpha^{\phi}_{i}=(\ell^{\phi}_{i},\beta^{\phi}_{i})$ to refer to a control $\phi \in \Phi^{\circ,m}$. 
\vs2
    
We allow for not observing nor being able to act on the system before  a random time $\vartheta^{\phi}_{i}$ defined by 
 \[
    \vartheta^{\phi}_{i} := \varpi(\tau^{\phi}_{i},X^{\phi}_{\tau^{\phi}_{i}-},\alpha^{\phi}_{i}, \upsilon, \epsilon_{i}),
  \]
where  $X^{\phi}$ is the controlled state process that will be described below, and 
\be\label{eq: hype varpi}
\varpi :\R_{+}\x \R^{d}\x \Ab \times \U \times \Ec \rightarrow [0, T]\;\text{ is measurable, such that $\varpi(t,\cdot)\ge t$ for all $t\ge 0$. }
\ee  
In the case where the actions consist in launching a trading robot at $\tau_{i}^{\phi}$ during a certain time $\ell^{\phi}_{i}$,  we can naturally take $ \vartheta^{\phi}_{i}=\tau^{\phi}_{i}+\ell^{\phi}_{i}$. If the action consists in placing a limit order during a maximal duration $\ell^{\phi}_{i}$, $\vartheta^{\phi}_{i}$ is the time at which the limit order is executed if it is less than $\tau^{\phi}_{i}+\ell^{\phi}_{i}$, and $\tau^{\phi}_{i}+\ell^{\phi}_{i}$ otherwise. 
\vs2

We say that $\phi \in \Phi^{\circ,m}$ belongs to $\Phi^{m}$ if 
$
\vartheta^{\phi}_{i}\le \tau^{\phi}_{i+1}\; \mbox{ and }\; \tau^{\phi}_{i}<\tau^{\phi}_{i+1}$  $\P_{m}$-a.s. for all $i\ge 1$, 
and define 
\be\label{eq: def Nphi}
\Nc^{\phi}:=\left[\cup_{i\ge 1}[\tau^{\phi}_{i},\vartheta^{\phi}_{i})\right]^{c}.
\ee

We are now in a position to describe our controlled state process. 
Given some initial data $z:=(t, x) \in \Zb := [0,T]\x \R^{d}$,  and  $\phi \in \Phi^{m}$, we let $X^{z,\phi}$ 
be the unique strong solution on $[t,2T]$ of 
\begin{align}
		X = x&+  \left( \int_{t}^{\cdot}\1_{\Nc^{\phi}}(s)\mu\left(s,X_s\right)ds + \int_{t}^{\cdot}\1_{\Nc^{\phi}}(s)\sigma\left(s,X_{s}\right)dW_{s}   \right) \nonumber       \\
		&+  \sum_{i\ge 1}   \1_{\{t\le \vartheta^{\phi}_{i}\le \cdot\}} [ F(\tau^{\phi}_{i},X_{\tau^{\phi}_{i}-},   \alpha^{\phi}_{i}, \upsilon, \epsilon_{i})-X_{\tau^{\phi}_{i}-}]  .\label{eq: dyna X} 
\end{align}
In the above, the function 
\be\label{eq: hyp mu sigma F}
\begin{array}{c}
(\mu,\sigma,F) :  \R_{+}\times \R^{d} \x \Ab \times \U \times \Ec\mapsto  \R^{d}\times \mathbb{M}^{d} \times \R^{d} \;\text{ is measurable.}
\\
\text{ The map $(\mu,\sigma)$ is  continuous, and   Lipschitz with linear growth  }\\
\text{in its second argument, uniformly in the first one,}
\end{array}
\ee
with $\mathbb{M}^{d}$ defined as the set of $d\times d$ matrices. This dynamics means the following. When no action is currently made on the system, i.e.~on the intervals in  $\Nc^{\phi}$, the system evolves according to a stochastic differential equation  driven by the Brownian motion $W$:
$$
dX_{s}=\mu\left(s,X_s\right)ds+\sigma\left(s,X_{s}\right)dW_{s} \;\;\mbox{ on $\Nc^{\phi}$}.
$$
When an impulse is made at $\tau^{\phi}_{i}$, we freeze the dynamics up to the end of the action at time $\vartheta^{\phi}_{i}$. This amounts to saying that we do not observe the current evolution up to $\vartheta^{\phi}_{i}$. At the end of the action, the  state process takes a new value 
$
X_{\vartheta^{\phi}_{i}}=F(\tau^{\phi}_{i}, X_{\tau^{\phi}_{i}-},  \alpha^{\phi}_{i}, \upsilon, \epsilon_{i}). 
$
The fact that $F$ depends on the unknown parameter $\upsilon$ and the additional noise $\epsilon_{i}$ models the fact the correct model  is not known with certainty, and that the exact value of the unknown parameter $\upsilon$ can (possibly) not be measured precisely just by observing  $(\vartheta^{\phi}_{i}-\tau^{\phi}_{i},X_{\vartheta^{\phi}_{i}}-X_{\tau^{\phi}_{i}-})$. 
 \vs2
 
 In order to simplify the notations, we shall now write:
 \begin{align}\label{eq: def Z Zcirc}
 Z^{z,\phi}:=(\cdot,X^{z,\phi})\;\;\mbox{ and }\;\;
 Z^{z,\circ}:=(\cdot,X^{z,\circ})
 \end{align}
 in which $X^{z,\circ}$ denotes the solution of \reff{eq: dyna X} for $\phi$ such that $\tau_{1}^{\phi}>T$ and satisfying $X_{t}^{z,\circ}=x$. This corresponds to the stochastic differential equation \reff{eq: dyna X} in the absence of impulse. Note in particular that 
 \begin{align}\label{eq: flow prop X} 
 &Z^{z,\phi}_{\vartheta^{\phi}_{1}}=\zr'(Z^{z,\circ}_{\tau^{\phi}_{1}-} ,\alpha^{\phi}_{1}, \upsilon, \epsilon_{1})\;\; \mbox{ on $\{\tau^{\phi}_{1}\ge t\}$,}
 \;\;\;\mbox{with }\;\; {\rm z}':=(\varpi,F).
\end{align} 
 
 From now on, we denote by $\F^{z,m,\phi}=(\Fc^{z,m,\phi}_{s})_{t\le s\le 2T}$ the $\P_{m}$-augmentation of the filtration generated by 
$(X^{z,\phi},\sum_{i\ge 1}\1_{[\vartheta_{i}^{\phi},\infty)})
$  
on $[t,2T]$. We say that $\phi \in \Phi^{m}$ belongs to $\Phi^{z,m}$ if $(\tau_{i}^{\phi})_{i\ge 1}$ is a sequence of $\F^{z,m,\phi}$-stopping times  and $\alpha^{\phi}_{i}$ is $\Fc^{z,m,\phi}_{\tau^{\phi}_{i}}$-measurable, for each $i\ge 1$.  Hereafter an admissible control will be an element of $\Phi^{z,m}$.

\subsection{Bayesian updates}

Obviously, the prior $m$ will evolve with time, as the value of the unknown parameter is partially revealed through the observation of the impacts of the actions on the system: at time $t$, one has observed  $\{\zr'(Z^{z,\phi}_{\tau^{\phi}_{i}-} ,\alpha^{\phi}_{i}, \upsilon, \epsilon_{i}):$ $i\ge 1, \vartheta^{\phi}_{i}\le t\}$.
 It should therefore be considered as a state variable, in any case, as its dynamics will naturally appear in any dynamic programming principle related to the optimal control of $X^{z,\phi}$, see Proposition \ref{prop: DPP} below.   
 Moreover, its evolution can be of interest in itself. One can for instance be interested by the precision of our (updated) prior at the end of the control period, as it can serve as a new prior for another control problem.
 \vs2 
 
 In this section, we   describe   how it is updated with time, according to the usual Bayesian procedure. 
Given $z=(t,x)\in \Zb$, $u\in \U$ and $a\in \Ab$, we assume that the law {under $\mathbb{P}_{\epsilon}$} of ${\rm z}'[z,a,u,\epsilon_{1}]$, recall \reff{eq: flow prop X}, 
is given by 
	$
		\qr(\cdot| z, a, u)d\Qr(\cdot|z, a),
	$
in which $\qr(\cdot| \cdot)$ is a Borel measurable map and $\Qr(\cdot|z, a)$ is a dominating measure  on $\Zb$ for each $(z,a)\in \Zb\x \Ab$.
For $z=(t,x)\in \Zb$, $m\in \Mb$ and $\phi\in \Phi^{z,m}$, let $M^{z,m,\phi}$ be the   process defined by 
\be\label{eq: def M}
M^{z,m,\phi}_{s}[C]:=\P_{m}[\upsilon \in C|\Fc^{z,m,\phi}_{s}],\;\;{ C\in \Bc(\U) }, \;s\ge t. 
\ee
As  no new information is revealed in between the end of an action and the start of the next one, the prior should remain constant on these time intervals: 
\begin{align}\label{eq: dyna m out of vartheta}
M^{z,m,\phi}=M^{z,m,\phi}_{\vartheta^{\phi}_{i}} \mbox{ on } [\vartheta^{\phi}_{i},\tau^{\phi}_{i+1}) \;,\;\;i\ge 0,
\end{align}
with the conventions $\vartheta^{\phi}_{0}=0$ and $M^{z,m,\phi}_{0}=m$. But, $M^{z,m,\phi}$ should jump from each $\tau^{\phi}_{i}$ to each $\vartheta^{\phi}_{i}$, $i\ge 1$, according to the Bayes rule:
\begin{align} 
M^{z,m,\phi}_{\vartheta^{\phi}_{i}}  &={\mathfrak M}(M^{z,m,\phi}_{\tau^{\phi}_{i}-};Z^{z,\phi}_{\vartheta^{\phi}_{i}},Z^{z,\phi}_{\tau^{\phi}_{i}-},\alpha^{\phi}_{i}),\;\;i\ge 1,\label{eq: dyna m}
\end{align}
in which 
\be\label{eq: def M mathfrak}
{\mathfrak M}(m_{o};z'_{o},z_{o},a_{o})[C]:=\frac{\int_{C}  \qr(z'_{o}|z_{o},a_{o}, u) dm_{o}(u)}
{\int_{\U}  \qr(z'_{o}|z_{o},a_{o}, u) dm_{o}(u)},
\ee
for almost all $(z_{o},z'_{o},a_{o},m_{o})\in \Zb^{2}\x \Ab\x \Mb$ and $C\in \Bc(\U)$. 
\vs2

Note that we did not {specify} $M^{z, m, \phi}$ on each $[\tau^{\phi}_{i}, \vartheta^{\phi}_{i})$ since the controller must wait  until $\vartheta^{\phi}_{i}$ before being able to make another action. A partial information on $\upsilon$ through $\vartheta^{\phi}_{i}$ is known as a right-censored observation of $\vartheta^{\phi}_{i}$ is revealed through the interval $[\tau^{\phi}_{i}, \vartheta^{\phi}_{i})$.

In order to ensure that $M^{z,m,\phi}$ remains in $\Mb$ whenever $m\in \Mb$, {we need the following standing assumption:}
\begin{Assumption}[Standing Assumption]
$$
{{\mathfrak M}(\Mb;\cdot)\subset \Mb.}
$$
\end{Assumption}

\begin{Remark}\label{rem: loc compact} {The above assumption means that we have to define a locally compact space $\Mb$ such  the initial prior belongs to $\Mb$, and that is stable under the operator ${\mathfrak M}$. It is important for the use of viscosity solutions. This is clearly a limitation of our approach, from a theoretical point of view. An alternative would be to lift $\Mb$ to the space of square integrable random variables, and then use the methodologies developped in the context of mean-field games (see e.g.~\cite[Section 6]{carda12}). We prevent from doing this for sake of clarity. On the other hand, our assumptions are satisfied in many pratical applications where $\Mb$ is either  a set  of measures defined on a metrizable compact space, see e.g.~\cite[Proposition 7.22 p130]{dimitri1996stochastic}, or a parameterized family (which needs to be the case eventually if a  numerical resolution is performed). If it is a parameterized family, it suffices  to find an homeomorphism $f$ from an open set of $\R^{k}$, $k\ge 1$, to $\Mb$ to ensure that $\Mb$ is locally compact. On the other hand, the stability of $\Mb$ with respect to ${\mathfrak M}$ can be ensured by using conjugate families, as explained in e.g.~\cite[Chapter 5.2]{bernardo2001bayesian}. The simplest example being the convex hull of a family of Dirac masses.  See  \cite{BBD16extended} for examples of applications.}
\end{Remark}

We formalize the dynamics of $M^{z,m,\phi}$   in the next proposition.   

\begin{Proposition}\label{prop: dyna M} For  all $z=(t,x)\in \Zb$, $m\in \Mb$ and $\phi\in \Phi^{z,m}$, the process $M^{z,m,\phi}$ is $\Mb$ valued and follows the dynamics  \reff{eq: dyna m out of vartheta}-\reff{eq: dyna m} on $[t,2T]$. 
\end{Proposition}

\begin{proof}   Let $C$ be a Borel set of $\U$ and $\vp$ be a Borel bounded function on the Skorohod space ${\rm D^{d+1}}$ of c\`{a}dl\`{a}g\footnote{continue \`{a} droite et limit\'{e}e \`{a} gauche (right continuous with left limits)} functions with values in $\R^{d+1}$. Set $\xi^{\phi}:=\sum_{i\ge 1}\1_{[\vartheta_{i}^{\phi},\infty)}$ and set $\delta X^{i}:= X_{\cdot\vee \vartheta^{\phi}_{i}}^{z, \phi}-X_{\vartheta^{\phi}_{i}}^{z, \phi}$. One can find a Borel measurable map $\bar \vp$ on ${\rm D^{2d+1}}$ such that 
$$
\vp(  X^{z,\phi}_{\cdot\wedge s},\xi^{\phi}_{\cdot\wedge s})\1_{\{\vartheta^{\phi}_{i}\le s<  \tau^{\phi}_{i+1}\}}=\bar \vp(X^{z,\phi}_{\cdot\wedge \vartheta^{\phi}_{i}}, \delta X^{i}_{\cdot \wedge s},\xi^{\phi}_{\cdot\wedge \vartheta^{\phi}_{i}})\1_{\{\vartheta^{\phi}_{i}\le s<  \tau^{\phi}_{i+1}\}}. 
$$
Then, the independence of $\upsilon$ with respect to $\sigma(W_{\cdot\vee \vartheta^{\phi}_{i}}-W_{\vartheta^{\phi}_{i}})$ given $\Fc^{z,m,\phi}_{\vartheta^{\phi}_{i}}$, and the fact that $\tau^{\phi}_{i+1}$ is measurable with respect to the sigma-algebra generated by $\sigma(W_{\cdot\vee \vartheta^{\phi}_{i}}-W_{\vartheta^{\phi}_{i}})$ and $\Fc^{z,m,\phi}_{\vartheta^{\phi}_{i}}$ imply that, for $s\ge 0$,  
\begin{align*}
\E_{m}\left[\1_{\{\upsilon\in C\}}\vp(  X^{z,\phi}_{\cdot\wedge s},\xi^{\phi}_{\cdot\wedge s})\1_{\{\vartheta^{\phi}_{i}\le s< \tau^{\phi}_{i+1}\}} \right]
&=\E_{m}\left[\1_{\{\upsilon\in C\}}\bar \vp(X^{z,\phi}_{\cdot\wedge \vartheta^{\phi}_{i}}, \delta X^{i}_{\cdot \wedge s},\xi^{\phi}_{\cdot\wedge \vartheta^{\phi}_{i}})\1_{\{\vartheta^{\phi}_{i}\le s<  \tau^{\phi}_{i+1}\}} \right]\\
&=\E_{m}\left[ M^{z,m,\phi}_{\vartheta^{\phi}_{i}}[C]\bar \vp(X^{z,\phi}_{\cdot\wedge \vartheta^{\phi}_{i}}, \delta X^{i}_{\cdot \wedge s},\xi^{\phi}_{\cdot\wedge \vartheta^{\phi}_{i}})\1_{\{\vartheta^{\phi}_{i}\le s<  \tau^{\phi}_{i+1}\}} \right]\\
&=\E_{m}\left[ M^{z,m,\phi}_{\vartheta^{\phi}_{i}}[C]\vp(X^{z,\phi}_{\cdot\wedge s},\xi^{\phi}_{\cdot\wedge s})\1_{\{\vartheta^{\phi}_{i}\le s<  \tau^{\phi}_{i+1}\}} \right].
\end{align*}
This shows  that $M^{z,m,\phi}_{s}[C]\1_{\{\vartheta^{\phi}_{i}\le s<  \tau^{\phi}_{i+1}\}}=M^{z,m,\phi}_{\vartheta^{\phi}_{i}}[C]\1_{\{\vartheta^{\phi}_{i}\le s<  \tau^{\phi}_{i+1}\}}$ $\P_{m}-\as$

It remains to compute $M^{z,m,\phi}_{\vartheta^{\phi}_{i}}$.  Note that \reff{eq: dyna X} implies that  $(X^{z,\phi}_{\tau^{\phi}_{i}-},$ $\xi^{\phi}_{\tau^{\phi}_{i}-})$ $=$ $(X^{z,\phi}_{\vartheta^{\phi}_{i}-},$ $\xi^{\phi}_{\vartheta^{\phi}_{i}-})$. Let $\vp$ be as above, and let  $\bar \vp$ be a Borel measurable map on ${\rm D^{d+1}}\x \R_{+}\x \R^{d}$ such that 
\begin{align*}
\vp(  X^{z,\phi}_{\cdot\wedge \vartheta^{\phi}_{i}},\xi^{\phi}_{\cdot\wedge \vartheta^{\phi}_{i}})&=\bar \vp(X^{z,\phi}_{\cdot\wedge \tau^{\phi}_{i}-},\xi^{\phi}_{\cdot\wedge \tau^{\phi}_{i}-},  \vartheta^{\phi}_{i},X^{z,\phi}_{\vartheta^{\phi}_{i}})=\bar \vp(X^{z,\phi}_{\cdot\wedge \tau^{\phi}_{i}-},\xi^{\phi}_{\cdot\wedge \tau^{\phi}_{i}-},  
{\rm z}'[\tau^{\phi}_{i},X^{z,\phi}_{\tau^{\phi}_{i}-},\alpha^{\phi}_{i},\upsilon,\epsilon_{i}] ).
\end{align*}
Then, since $\epsilon_{i}$ is independent of $\Fc^{z,m,\phi}_{\tau^{\phi}_{i}}$ and has the same law as $\epsilon_{1}$, 
\begin{align*}
&\E_{m}\left[\1_{\{\upsilon\in C\}}\vp(  X^{z,\phi}_{\cdot\wedge \vartheta^{\phi}_{i}},\xi^{\phi}_{\cdot\wedge \vartheta^{\phi}_{i}})\right]
\\&=\E_{m}\left[\1_{\{\upsilon\in C\}}\bar \vp(X^{z,\phi}_{\cdot\wedge \tau^{\phi}_{i}-},\xi^{\phi}_{\cdot\wedge \tau^{\phi}_{i}-},  
{\rm z}'[\tau^{\phi}_{i},X^{z,\phi}_{\tau^{\phi}_{i}-},\alpha^{\phi}_{i},\upsilon,\epsilon_{i}] )\right]\\
&=\E_{m}\left[\int \1_{\{\upsilon\in C\}}\bar \vp(X^{z,\phi}_{\cdot\wedge \tau^{\phi}_{i}-},\xi^{\phi}_{\cdot\wedge \tau^{\phi}_{i}-},  
z') \qr(z'|Z^{z,\phi}_{\tau^{\phi}_{i}-},\alpha^{\phi}_{i},\upsilon)d\Qr(z'|Z^{z,\phi}_{\tau^{\phi}_{i}-},\alpha_{i}^{\phi}) )\right]\\
&=\E_{m}\left[\int \bar \vp(X^{z,\phi}_{\cdot\wedge \tau^{\phi}_{i}-},\xi^{\phi}_{\cdot\wedge \tau^{\phi}_{i}-},  
z') \left( \int_{C} \qr(z'|Z^{z,\phi}_{\tau^{\phi}_{i}-}, \alpha_{i}^{\phi},u)dM^{z,m,\phi}_{\tau_{i}^{\phi}-}(u)\right)d\Qr(z'|Z^{z,\phi}_{\tau^{\phi}_{i}-},\alpha^{\phi}_{i}) )\right].
\end{align*}
Let us now introduce the notation 
$
{\mathfrak M}_{i}[C](z'):={\mathfrak M}(M^{z,m,\phi}_{\tau^{\phi}_{i}-};z',Z^{z,\phi}_{\tau^{\phi}_{i}-},\alpha^{\phi}_{i}).
$
Then, 
\begin{align*}
&\E_{m}\left[\1_{\{\upsilon\in C\}}\vp(  X^{z,\phi}_{\cdot\wedge \vartheta^{\phi}_{i}},\xi^{\phi}_{\cdot\wedge \vartheta^{\phi}_{i}})\right]
\\
&=\E_{m}\left[\int \bar \vp(X^{z,\phi}_{\cdot\wedge \tau^{\phi}_{i}-},\xi^{\phi}_{\cdot\wedge \tau^{\phi}_{i}-},  
z'){\mathfrak M}_{i}[C](z')
 \qr(z'|Z^{z,\phi}_{\tau^{\phi}_{i}-},\alpha_{i}^{\phi},\upsilon)d\Qr(z'|Z^{z,\phi}_{\tau^{\phi}_{i}-},\alpha_{i}^{\phi}) )\right]\\
&=\E_{m}\left[   \vp(X^{z,\phi}_{\cdot\wedge \vartheta^{\phi}_{i}},\xi^{\phi}_{\cdot\wedge \vartheta^{\phi}_{i}}){\mathfrak M}_{i}[C](Z^{z,\phi}_{\vartheta^{\phi}_{i}})\right].
\end{align*}
This concludes the proof.\ep\end{proof}


 \begin{Remark}\label{rem: joint condi distri ZM}
For later use, note  that the above provides  the joint conditional distribution of  $(Z^{z,\phi}_{\vartheta^{\phi}_{i}},M^{z,m,\phi}_{\vartheta^{\phi}_{i}})$ given ${\Fc^{z,m,\phi}_{\tau_{i}-}}$. Namely, for Borel sets $B \in \mathcal{B}([t,T]{\x}\mathbb{R}^{d})$ and $D \in \mathcal{B}(\Mb)$, a simple application of Fubini's Lemma implies that 
 \begin{equation}\label{eq: joint cond law X,M}
\P[(Z^{z,\phi}_{\vartheta^{\phi}_{i}},M^{z,m,\phi}_{\vartheta^{\phi}_{i}})\in B\x D |{\Fc^{z,m,\phi}_{\tau^{\phi}_{i}-}}]= \kr(B\x D|Z^{z,\phi}_{\tau^{\phi}_{i}-}, M^{z,m\phi}_{\tau^{\phi}_{i}-},\alpha^{\phi}_{i})
 \end{equation}
 in which
 \be\label{eq: def nu}
 \kr(B\x D| z_{o},m_{o},a_{o}):=
\int_{\U} \int_{B}  \1_{D}({\mathfrak M}(m_{o};z',z_{o},a_{o})) \qr(z'|z_{o},a_{o}, u)d\Qr(z'|z, a) dm_{o}(u),\;\;
 \ee
 for $(z_{o},m_{o},a_{o})\in \Zb\x \Mb\x \Ab.$  
\end{Remark}

	\subsection{Gain function}

Given $z=(t,x) \in \Zb$ and $m\in \Mb$, the aim of the controller is to maximize the expected value of the gain functional
    \[
        \phi \in \Phi^{z,m} \mapsto {G}^{z,m}(\phi) := g(Z^{z,\phi}_{\T[\phi]}, M^{z,m,\phi}_{\T[\phi]}, \upsilon, \epsilon_{0}),
    \]
 in which $\T[\phi]$ is the end of the last action after $T$:
 $$
 \T[\phi]:=\sup\{\vartheta^{\phi}_{i}: i\ge 1,\; \tau^{\phi}_{i}\le T\}\vee T.
 $$   
 As suggested earlier,  the gain may not only depend on the value of the original time-space state process $Z^{z,\phi}_{\T[\phi]}$ but also on $M^{z,m,\phi}_{\T[\phi]}$, to model the fact that we are also interested by the precision of the estimation made on $\upsilon$ at the final time. One also allows for terminating the last action after $T$. However, since $g$ can depend on  $ \T[\phi]$ through $Z^{z,\phi}_{\T[\phi]}$, one can penalize the actions that actually terminates strictly after $T$. 
 \vs2
 
Hereafter, the function $g$ is assumed to be measurable and  bounded\footnote{Boundedness is just for sake of simplicity. Much more general frameworks could easily be considered.} on $\Zb \times \Mb \times \U\times \Ec$. 
\vs2

Given $\phi \in \Phi^{z,m}$, the expected gain is 
    \[
        J(z, m; \phi) := \mathbb{E}_{m}\left[ G^{z,m}(\phi)\right],
    \]
and 
    \begin{align}\label{eq: def vr}
        \vr(z,m) := \sup_{\phi \in \Phi^{z,m}}J(z, m; \phi)\1_{\{t\le T\}}+\1_{\{t> T\}}\mathbb{E}_{m}\left[ g(z, m, \upsilon, \epsilon_{0})\right] 
    \end{align}
   is the corresponding value function. Note that $\vr$ depends on $m$ through the set of admissible controls $ \Phi^{z,m}$ and the expectation operator $ \mathbb{E}_{m}$, even if $g$ does not depend on $M^{z,m,\phi}_{\T[\phi]}$. 
   
 \begin{Remark}  Note that a running gain term could be added without any difficulty. One usually reduces to a Mayer formulation by adding a component to the space process and by modifying the terminal reward accordingly. Here,   if this running gain  only covers the period $[0,T]$, it should be added explicitely because of the modified time horizon $\T[\phi]$ at which the terminal gain is computed. 
   \end{Remark}
  \section{Value function characterization}

The aim of this section is to provide a characterization of the value function $\vr$. 
As usual, it should be related to a  dynamic programming principle. In our setting, it  corresponds to: Given $z=(t,x) \in \Zb$ and $m\in \Mb$, then 
	\begin{equation}\label{eq: DPP formal}
		\vr(z,m)= \sup_{\phi \in \Phi^{z,m}}\E_{m} [\vr(Z^{z,\phi}_{\theta^{\phi}},M^{z,m,\phi}_{\theta^{\phi}})],
	\end{equation}
for all collection $(\theta^{\phi},\phi\in \Phi^{z,m})$  of $\F^{z,m,\phi}$-stopping times with values in $[t, 2T]$ such that 
$
\theta^{\phi}\in \Nc^{\phi}\cap [t,\T[\phi]]~\P_{m}-\as,
$
recall the definition of $\Nc^{\phi}$ in \reff{eq: def Nphi}. 

Let us comment this. First, one should restrict to stopping times such that $\theta^{\phi}\in \Nc^{\phi}$. The reason is that no new impulse can be made outside of $\Nc^{\phi}$, each interval $[\tau_{i}^{\phi},\vartheta^{\phi}_{i})$ is a latency period. Second, the terminal gain is evaluated at $\T[\phi]$, which in general is different from $T$. Hence, the fact that $\theta^{\phi}$ is only bounded by $\T[\phi]$. 

A partial version of \reff{eq: DPP formal} will be proved in Proposition \ref{prop: DPP} below and will be used to provide a sub-solution property. As  already mentioned in the introduction, we are not able to prove a full version \reff{eq: DPP formal}.  The reason is that the value function $\vr$ depends on $z=(t,x)\in \Zb$ and $m\in \Mb$ through the set of admissible controls $\Phi^{z,m}$, and more precisely through the choice of the filtration $\F^{z,m,\phi}$, which even depends on $\phi$ itself. This makes this dependence highly singular and we are neither in position to play with any a-priori smoothness, see e.g.~\cite{bouchard2011weak}, nor to apply a measurable selection theorem, see e.g.~\cite{dimitri1996stochastic}. 

\vs2
We continue our discussion, assuming that \reff{eq: DPP formal} holds and that $\vr$ is sufficiently smooth.
Then, it should in particular satisfy
$
\vr(z,m)\geq  \E_{m} [\vr(Z^{z,\circ}_{t+h},m)]
$ 
whenever $z=(t,x)\in [0,T)\x \R^{d}$ and $0<h\le T-t$ ($Z^{z,\circ}$ is defined after \reff{eq: def Z Zcirc}). This corresponds to the sub-optimality   of the control consisting in making no impulse on  $[t,t+h]$. Applying It\^{o}'s lemma, dividing by $h$ and letting $h$ go to $0$, we obtain 
$
-\Lc\vr(z,m)\ge 0 
$
in which  
 $\Lc$ is the Dynkin operator associated to $X^{z,\circ}$, 
 $$
 \Lc\vp:=\partial_{t}\vp  +\langle \mu, D\vp\rangle + \frac12 {\rm Tr}[\sigma\sigma^{\top} D^{2}\vp]. 
 $$
 On the other hand, it follows from \reff{eq: DPP formal} and Remark \ref{rem: joint condi distri ZM} that 
\begin{align}
 &\vr(z,m)\geq  \sup_{a\in \Ab}\E_{m} [\vr({\rm z}'[z,a,\upsilon,\epsilon_{1}],{\mathfrak M}(m;{\rm z}'[z,a,\upsilon,\epsilon_{1}],z,a)) ] = \Kc\vr(z,m)
\nonumber\\
&\mbox{where}\;\;\label{eq : def Kc}
\Kc \vp:=\sup_{a\in \Ab}\Kc^{a}\vp \;\;\mbox{ with } \;\;
\Kc^{a} \vp:= \int \vp(z',m')d\kr(z',m'|\cdot,a) \mbox{ for $a\in \Ab$.}
 \end{align}
 As for the time-$T$ boundary condition, the same reasoning as above implies 
 $
 \vr(T,\cdot)\ge \Kc_{T}g$ and $\vr(T,\cdot)\ge \Kc\vr(T,\cdot), 
 $
 in which 
\begin{align}\label{eq : def KcT}
 \Kc_{T}g(\cdot,m)=\int_{\U}\int_{\Ec} g(\cdot,m,u,e)d\P_{\epsilon}(e)dm(u).
\end{align} 
 
 By optimality, $\vr$ should therefore solve the quasi-variational equations
 \begin{align}
 \min\left\{-\Lc\vp\;,\;\vp-\Kc\vp\right\}=0 &\;\mbox{ on } [0,T)\x \R^{d}\x\Mb\label{eq: pde interior}\\
  \min\left\{\vp- \Kc_{T}g,\vp- \Kc \vp\right\}=0 &\;\mbox{ on } \{T\}\x \R^{d}\x\Mb,\label{eq: pde T} 
  \end{align}
  in the sense of the following definition (given for sake of clarity). 
  
  \begin{Definition} We say that a lower-semicontinuous function $U$ on $\R_{+}\x \R^{d}\x\Mb$ is a viscosity super-solution of \reff{eq: pde interior}-\reff{eq: pde T}  if for any $z_{\circ}=(t_{\circ},x_{\circ})\in \Zb$, $m_{\circ}\in \Mb$, and $\vp \in C^{1,2,0}([0,T]\x \R^{d}\x \Mb)$ such that  $\min_{\Zb\x \Mb}(U-\vp)$ $=$ $(U-\vp)(z_{\circ},m_{\circ})$ $=$ $0$ we have 
  $$
 \left[  \min\left\{-\Lc\vp\;,\;\vp-\Kc U\right\}\1_{\{t_{\circ}<T\}} + \min\left\{\vp- \Kc_{T}g,\vp- \Kc U\right\}\1_{\{t_{\circ}=T\}}\right](z_{\circ},m_{\circ})\ge 0.
  $$  
   We say that a upper-semicontinuous function $U$ on $\R_{+}\x \R^{d}\x \Mb$ is a viscosity sub-solution of \reff{eq: pde interior}-\reff{eq: pde T}  if for any $z_{\circ}=(t_{\circ},x_{\circ})\in \Zb$, $m_{\circ}\in \Mb$ and  $\vp \in C^{1,2,0}([0,T]\x \R^{d}\x \Mb)$ such that  $\max_{\Zb\x \Mb}(U-\vp)$ $=$ $(U-\vp)(z_{\circ},m_{\circ})$ $=$ $0$ we have
  $$
 \left[  \min\left\{-\Lc\vp\;,\;\vp-\Kc U\right\}\1_{\{t_{\circ}<T\}} + \min\left\{\vp- \Kc_{T}g,\vp- \Kc U\right\}\1_{\{t_{\circ}=T\}}\right](z_{\circ},m_{\circ})\le 0.
  $$  
We say that a continuous function $U$ on $\R_{+}\x \R^{d}\x\Mb$ is a viscosity solution of \reff{eq: pde interior}-\reff{eq: pde T} if it is a super- and a sub-solution.
  \end{Definition}
  
To ensure that the above operator is continuous, we assume from now on that, on  $\R_{+}\x \R^{d}\x \Mb$, 
\be\label{eq: hyp conti KcT et Kc}
\begin{array}{c}
\text{$\Kc_{T}g$ is continuous, and $\Kc\vp$ is upper- (resp.~lower-) semicontinuous,}\\
\text{for all   upper- (resp.~lower-) semicontinuous bounded function $\vp$.}
\end{array}
\ee

A sufficient condition for (\ref{eq: hyp conti KcT et Kc}) to hold is that $\kr$ defined in (\ref{eq: def nu}) is a continuous stochastic kernel, see \cite[Proposition 7.31 and 7.32 page 148]{dimitri1996stochastic}. 

Finally, we assume that comparison holds for \reff{eq: pde interior}-\reff{eq: pde T}.

\begin{Assumption}\label{ass: comp}
Let $U$ (resp.~$V$) be a upper- (resp.~lower-) semicontinuous bounded viscosity sub- (resp.~super-) solution of  \reff{eq: pde interior}-\reff{eq: pde T}. Assume further that $U \le V$ on $(T,\infty)\x \R^{d}\x \Mb$.
Then, $U \le V$ on $\Zb\x \Mb$.
\end{Assumption}

See Proposition \ref{comparaison} below for a sufficient condition. We are now in position to state the main result of this paper.    The proof is provided in the next section.  

\begin{Theorem}\label{thm:viscosity}
Let Assumption \ref{ass: comp}  (or the conditions of Proposition \ref{comparaison} below) hold. Then,  $\vr$ is continuous on $\Zb\x \Mb$ and is the unique bounded viscosity solution of \reff{eq: pde interior}-\reff{eq: pde T}.
\end{Theorem}

\begin{Remark} {We do not discuss here the issue of existence of an optimal control. We refer to the application paper \cite{BBD16extended} for an example of numerical scheme allowing to construct approximately optimal controls. Note also that the construction of Section \ref{sec: proof DPP vrn} below produces an almost optimal control as the arguments of Section \ref{sec: conc proof super sol} show that the sequence of value functions  $(\vr_{n})_{n\ge 1}$ actually converges to $\vr$. }
\end{Remark}
    \section{Viscosity solution properties}\label{sec: proof disco}

This part is dedicated to the proof of the viscosity solution characterization of Theorem \ref{thm:viscosity}. 
We start with the sub-solution property, which is the more classical part. As for the super-solution property, we shall later on introduce a discrete time version of the model that will provide a natural lower bound. We will then show   that the sequence of corresponding value functions converges to a super-solution of our quasi-variational equation as  the time step goes to $0$. By comparison, we will finally identify this (limit) lower bound to the original value function, thus showing that the later is also a super-solution.   
\subsection{Sub-solution property}

We start with the sub-solution property and show that it is satisfied by the upper-semicontinuous enveloppe  of $\vr$ defined in  \reff{eq: def vr}: 
$$
\vr^{*}(z,m):=\limsup_{(z',m')\to (z,m)} \vr(z',m')\;\;,\;\;(z,m)\in  \R_{+}\x \R^{d}\x \Mb.
$$

\begin{Proposition}\label{prop: subsol} $\vr^{*}$ is a viscosity subsolution of \reff{eq: pde interior}-\reff{eq: pde T}.
\end{Proposition}

The proof  is rather standard. As usual, it is based on the  partial dynamic programming principle contained in Proposition \ref{prop: DPP} below,  that can be established by adapting  standard   lines of arguments, see e.g.~\cite{bouchard2011weak}. For this part, the dependency of the filtration on the initial data is not problematic as it only requires a conditioning argument. 
  Before to state it, let us make an observation.

 \begin{Remark}\label{rem : conditioning} Note that, given $z=(t,x)\in \Zb$, the process $X^{z,\circ}$  defined in \reff{eq: def Z Zcirc} is predictable with respect to the $\P$-augmentation of the  raw filtration $\F^{t,W}$ generated by $(W_{\cdot\vee t}-W_{t})$. By  \cite[Lemma 7, Appendix I]{DellacherieMeyer.82}, it is indistinguishable from a $\F^{t,W}$-predictable process. 
Using this identification, $X^{z,\circ}_{s}(\omega)=X^{z,\circ}_{s}(\omega^{t,s})$ for $s\ge t$, with $\omega^{t,s}:=\omega_{t\vee \cdot\wedge s}-\omega_{t}$. Similarly, $\tau^{\phi}_{1}$ and $\alpha^{\phi}_{1}$ can be identified to Borel measurable maps on $C([0,T];\R^{d})$ that depends only on $\omega^{t,\tau^{\phi}_{1}(\omega^{t,T})}$ so that $(Z^{z,\phi}_{\vartheta^{\phi}_{1}},M^{z,m,\phi}_{\vartheta^{\phi}_{1}})$ can be seen as a Borel map on  $C([0,T];\R^{d})\x \U\x \Ec$, while $(Z^{z,\phi}_{\tau^{\phi}_{1}-},M^{z,m,\phi}_{\tau^{\phi}_{1}-})$ can be seen as a Borel map on  $C([0,T];\R^{d})$ that only depends on $\omega^{t,\tau^{\phi}_{1}(\omega^{t,T})}$,  recall \reff{eq: flow prop X}, \reff{eq: dyna m out of vartheta} and \reff{eq: dyna m}. Iterating this argument, we also obtain that  $(Z^{z,\phi}_{\T[\phi]},M^{z,m,\phi}_{\T[\phi]})$ is equal, up to $\P_{m}$-null sets, to a  Borel map on $C([0,T];\R^{d})\x \U\x \Ec^{N}$, for some $N\ge 1$ that depends on $\phi$.   
 \end{Remark}

We use the notations introduced in \reff{eq: def Z Zcirc}, \reff{eq : def Kc} and \reff{eq : def KcT} in the following.
\begin{Proposition}\label{prop: DPP}
Fix $(z,m) \in \Zb \times \Mb$, and let  $\theta$ be the first exit time of $Z^{z,\circ}$  from a Borel set $B\subset \Zb$ containing $(z,m)$. Then,  
	\begin{align}\label{eq: ddp facile} 
		\vr(z,m)&\leq \sup_{\phi \in \Phi^{z,m}_{\ge t}}\E_{m} [f(Z^{z,\circ}_{\theta},m)\1_{\{\theta<\tau^{\phi}_{1}\}}+\Kc^{\alpha^{\phi}_{1}} f(Z^{z,\circ}_{\tau^{\phi}_{1}-},m)]\1_{\{\theta\ge \tau^{\phi}_{1}\}}]
	\end{align}
 in which $z:=(t,x)$, $\Phi^{z,m}_{\ge t}:=\{\phi \in \Phi^{z,m}: \tau^{\phi}_{1}\ge t\}$ and 
 \begin{align}\label{eq: def f}
 f(z',m'):=\vr^{*}(z',m')\1_{\{t'<T\}}+\Kc_{T}g(z',m')\1_{\{t'\ge T\}}
 \end{align}
 for $z'=(t',x')\in \Ab$ and $m'\in \Mb$.
\end{Proposition}
\proof Let $N\ge 1$ be such that $\tau^{\phi}_{i}>T$ for $i\ge N$. By right continuity of $(Z^{z,\phi},M^{z,m,\phi})$ and upper-semicontinuity of $f$ and $\Kc f$ on $[0,T)\x \R^{d}\x \Mb$, see \reff{eq: hyp conti KcT et Kc}, it suffices to prove the result for the projections on the right of $\theta$ and $\tau^{\phi}_{1}$ on a deterministic time grid. Then, it is enough to consider the case where $(\theta,\tau^{\phi}_{1})\equiv (s,s')\in [t,T]^{2}$, by arguing as below and conditioning by the values taken by $(\theta,\tau^{\phi}_{1})$ on the grid. In the following, we use regular conditional expectation operators. We shall make use of Remark \ref{rem : conditioning}. In particular, we write $\phi(\omega,u,(e_{i})_{i\le N})$ to denote the Borel map 
 $(\omega,u,(e_{i})_{i\le N})\in C([0,T];\R^{d})\x \U\x \Ec^{N}$ $\mapsto$ $\{(\tau^{\phi}_{i},\alpha^{\phi}_{i})(\omega^{t,T},u,(e_{j})_{j\le i-1}),$ $i\le N\}$ associated to $\phi$. 
If $s<s'$, we have $\P_{m}$-$\as$  
 \begin{align*}
 \E_{m}[G^{z,m}(\phi) | \Fc^{z,m,\phi}_{s}] (\omega,u,(e_{i})_{i\ge 1})
 &=
  \E_{m}[G^{Z^{z,\circ}_{s}(\omega^{t,s}),m}(\phi_{\omega^{t,s}})]
 \\&=
 \E_{m}[\Kc_{T}g(X^{Z^{z,\circ}_{s}(\omega^{t,s}),\phi_{\omega^{t,s}}}_{T},M^{Z^{z,\circ}_{s}(\omega^{t,s}),m,\phi_{\omega^{t,s}}}_{T}) ]
 \end{align*}
 in which $\Kc_{T}$ is defined in \reff{eq : def KcT} and 
 $$
 \phi_{\omega^{t,s}}:(\omega',u,(e_{i})_{i\le N})\in C([s,T];\R^{d})\x \U\x \Ec^{N} \mapsto \phi(\omega^{t,s}+\omega'_{\cdot\vee s}-\omega'_{s},u,(e_{i})_{i\le N})
 $$
 is an element of $\Phi^{Z^{z,\circ}_{s}(\omega^{t,s}),m,\phi_{\omega^{t,s}}}$. It follows that 
 $
  \E_{m}[G^{z,m}(\phi) | \Fc^{z,m,\phi}_{s}] \1_{s<s'}
  \le f(Z^{z,\circ}_{s},m)\1_{s<s'}$ $\P_{m}-\as
 $
Similarly, if $s\ge s'$, we have $\P_{m}$-$\as$  
 \begin{align*}
 \E_{m}[G^{z,m}(\phi) | \Fc^{z,m,\phi}_{s'-}] (\omega,u,(e_{i})_{i\le N})
 &=
  \E_{m}[G^{\xi(\omega^{t,s'},\upsilon,\epsilon_{1},\alpha^{\phi}_{1}(\omega^{t,s'}))}(\phi_{\omega^{t,s'}})]
 \end{align*}
with 
$$
\xi(\omega^{t,s'},\upsilon,\epsilon_{1},\alpha^{\phi}_{1}(\omega^{t,s'}))=\left(\cdot,
{\mathfrak M}(m;\cdot,Z^{z,\circ}_{s'-}(\omega^{t,s'}),\alpha^{\phi}_{1}(\omega^{t,s'}))\right)\circ\zr'(Z^{z,\circ}_{s'-}(\omega^{t,s'}),\alpha_{1}^{\phi}(\omega^{t,s'}),\upsilon,\epsilon_{1}),
$$
recall the notations in \reff{eq: flow prop X} and \reff{eq: def M mathfrak}.
Hence,  $\P_{m}$-$\as$, 
\begin{align*}
 \E_{m}[G^{z,m}(\phi) | \Fc^{z,m,\phi}_{s'-}] (\omega,u,(e_{i})_{i\le N})
 &\le
  \E_{m}[f(\xi(\omega^{t,s'},\upsilon,\epsilon_{1},\alpha^{\phi}_{1}(\omega^{t,s'})))]
  =   \Kc^{\alpha^{\phi}_{1}(\omega^{t,s'})} f(Z^{z,\circ}_{s'-}(\omega^{t,s'}),m),
 \end{align*} 
 in which $a\in \Ab\mapsto \Kc^{a}$ is defined in \reff{eq : def Kc}.
\ep

\vs2
 
\noindent{\bf Proof of Proposition \ref{prop: subsol}} As already mentioned, the proof is standard, we provide it for completeness. Let $  \vp$ be a (bounded) $C^{1,2,0}$ function and fix $(z_{\circ},m_{\circ})\in \Zb\x\Mb$ such that 
\begin{align}\label{eq: max strict}
0=(\vr^{*}-  \vp)(z_{\circ},m_{\circ})=\max_{\Zb\x \Mb}(\vr^{*}-  \vp).
\end{align}

We use the notation   $z_{\circ}=(t_{\circ},x_{\circ})\in [0,T]\x \R^{d}$. 

Step 1. We first assume that $t_{\circ}<T$. 
Let us suppose   that
    $
        \min\left\{- \Lc\vp  \;,\; \varphi-\Kc\vr^{*}\right\}(z_{\circ},m_{\circ}) > 0, 
    $
    and work towards a contradiction to Proposition \ref{prop: DPP}. Let ${\rm d}_{\Mb}$ be a metric compatible with the weak topology and let $\|\cdot\|_{\Zb}$ be the Euclidean norm on $\Zb$. We define 
$$
\bar \vp(z',m'):= \vp(z',m')+ \|z'-z_{\circ}\|^{4}_{\Zb}+{\rm d}_{\Mb}(m',m_{\circ}) .
$$
If the above holds, then
    $
        \min\left\{- \Lc\bar \vp  \;,\; \bar \varphi-\Kc\vr^{*}\right\}(z_{\circ},m_{\circ}) > 0.
    $
By our continuity assumption \reff{eq: hyp conti KcT et Kc}, we can find $\iota,\eta>0$, such that  
    \begin{align}\label{eq: proof sous sol strict sur sol}
        \min\left\{- \Lc\bar \vp  \;,\; \bar \varphi-\Kc\vr^{*}\right\} \ge \eta\;\;\mbox{ on } B_{\iota}, 
    \end{align}
    in which 
    $$
    B_{\iota}:=\{(z',m')\in \Zb\x \Mb: \|z'-z_{\circ}\|^{4}_{\Zb}+{\rm d}_{\Mb}(m',m_{\circ})< \iota\}\subset [0,T)\x \R^{d}\x \Mb.
    $$
 Note that, after possibly changing $\eta>0$, we can assume that 
     \begin{align}\label{eq: max strict bord}
        (\vr^{*}-\bar \vp)\le-\eta \mbox{ on }   (B_{\iota})^{c}.
    \end{align} 
In the following, we let  $(z,m)\in B_{\iota}$ be such that 
    \begin{align}\label{eq: dist vr}
        |\vr(z,m)-\bar \vp(z,m)|\le  \eta/2,
    \end{align}
    recall \reff{eq: max strict}. As above, we write $z=(t,x)\in [0,T]\x \R^{d}$. 
  Fix $\phi\in \Phi^{z,m}$. We write $(\tau_{i},\alpha_{i},\vartheta_{i})_{i\ge1}$, $Z$ and $M$ for  
  $(\tau^{\phi}_{i},\alpha^{\phi}_{i},\vartheta^{\phi}_{i})_{i\ge1}$, $Z^{z,\phi}$ and $M^{z,m,\phi}$.  Let $\theta$ be the first time when $(Z,M)$ exits $B_{\iota}$.
Without loss of generality, one can assume that $\tau_{1}\ge t$. Define 
$
\chi:=\theta\1_{\{\theta<\tau_{1}\}}+\1_{\{\theta\ge \tau_{1}\}}\vartheta_{1}.
$
In view of \reff{eq: proof sous sol strict sur sol}, \reff{eq: max strict bord} and \reff{eq: dist vr},  
 \begin{align*}
\E_{m}[\vr^{*}(Z_{\chi},M_{\chi})]&=  \E_{m}[\vr^{*}(Z_{\vartheta_{1}},M_{\vartheta_{1}})\1_{\{\chi\ne \theta\}}+\vr^{*}(Z_{\theta},M_{\theta})\1_{\{\chi=\theta\}}]
\\
&\le  \E_{m}[\Kc\vr^{*}(Z_{\tau_{1}-},M_{\tau_{1}-})\1_{\{\chi\ne \theta\}}+\vr^{*}(Z_{\theta},M_{\theta})\1_{\{\chi= \theta\}}]\\
&\le \E_{m}[\bar \vp(Z_{\theta \wedge \tau_{1}-},M_{\theta \wedge \tau_{1}-})]-\eta\\
&\le \bar  \vp(z,m)-\eta\\
&\le \vr(z,m)-\eta/2.
 \end{align*}
Since $\chi<T$, this  contradicts Proposition \ref{prop: DPP} by arbitrariness of $\phi$.

Step 2. We now consider the case $t_{\circ}=T$. We assume that 
    $
        \min\left\{\varphi-\Kc\vr^{*}  \;,\; \varphi-\Kc_{T}g\right\}(z_{\circ},m_{\circ}) > 0, 
    $
    and work toward a contradiction. 
     Let us  define 
$$
\bar \vp(t',x',m'):=\bar \vp(t',x',m')+C(T-t')+ \|(t',x')-z_{\circ}\|^{4}_{\Zb}+{\rm d}_{\Mb}(m',m_{\circ})
$$
and note that, for $C$ large enough, 
    $
        \min\left\{- \Lc\bar \vp  \;,\;\bar \varphi-\Kc\vr^{*}  \;,\; \bar \varphi-\Kc_{T} g\right\}(z_{\circ},m_{\circ}) > 0.
    $
    Then, as in Step 1, we can find $\iota,\eta>0$, such that  
    \begin{align*}
       \min\left\{- \Lc\bar \vp  \;,\;\bar \varphi-\Kc\vr^{*}  \;,\; \bar \varphi-\Kc_{T} g\right\} \ge \eta\;\;\mbox{ on } B_{\iota}, 
    \end{align*}
    in which 
    $$
    B_{\iota}:=\{(t',x',m')\in (T-\iota,T]\x \Mb: \|x'-x_{\circ}\|^{4}_{\R^{d}}+{\rm d}_{\Mb}(m',m_{\circ})< \iota\}.
    $$
    After possibly changing $\eta>0$, one can assume that 
     \begin{align*}
      (\vr^{*}-\bar \vp)\le-\eta \mbox{ on }   (B_{\iota})^{c}.
    \end{align*} 
 Let $(t,x,m)\in B_{\iota}$ be such that  
    \begin{align*}
        |\vr(t,x,m)-\bar \vp(t,x,m)|\le  \eta/2.
    \end{align*}
    One can assume that $t<T$. Otherwise, this would mean that 
    \begin{align*}
    \vr^{*}(z_{\circ},m_{\circ})=\limsup_{(T,x',m')\to (z_{\circ},m_{\circ})}\vr(T,x',m')=\limsup_{(T,x',m')\to (z_{\circ},m_{\circ})}\Kc_{T} (T,x',m')=\Kc_{T} g(z_{\circ},m_{\circ}),
    \end{align*}
    recall \reff{eq: hyp conti KcT et Kc},  and there is nothing to prove.     
    Given $\phi\in \Phi^{z,m}$, with $z:=(t,x)$, let $(\tau_{1},\vartheta_{1},Z=(\cdot,X),M)$ be defined as in Step 1 with respect to $\phi$ and $(z,m)$, and  consider 
  $
\chi:=\theta\1_{\{\theta<\tau_{1}\}}+\1_{\{\theta\ge \tau_{1}\}}\vartheta_{1},
$
where $\theta$ is the first exit time of $(X,M)$ from $\{(x',m')\in \R^{d}\x \Mb: \|x'-x_{\circ}\|^{4}_{\R^{d}}+{\rm d}_{\Mb}(m',m_{\circ})< \iota\}$. As in Step 1, the above implies that 
 $
\E_{m}[\vr^{*}(Z_{\chi},M_{\chi})]\le   \vr(z,m)-\eta/2,
$
which  contradicts Proposition \ref{prop: DPP} by arbitrariness of $\phi$.
\ep

\subsection{Discrete time approximation and dynamic programming}\label{sec: proof DPP vrn}

In this part, we prepare for the proof of the super-solution property.   As already mentioned above, we could not provide the opposite inequality in \reff{eq: ddp facile}, with $\vr^{*}$ replaced by the lower-semicontinuous envelope of $\vr$, because of the non-trivial dependence of  $\F^{z,m,\phi}$ with respect to the initial data. 
 Instead, we use the natural idea of approximating our continuous time control problem by a sequence of discrete time counterparts defined on a sequence of time grids. In discrete time, the dynamic programming principle can be proved along the lines of \cite{dimitri1996stochastic} for the corresponding value functions $(\vr_{n})_{n\ge 1}$. Passing to the limit as the time mesh vanishes provides a super-solution $\vr_{\circ}$ of \reff{eq: pde interior}-\reff{eq: pde T}. As $\vr^{*}$ is a sub-solution of the same equation,   Assumption \ref{ass: comp}  will imply that $\vr_{\circ}\ge \vr^{*}$, while the opposite will hold by construction. Then, we will conclude that $\vr$ is a actually a super-solution, and is even continuous.  This approach is similar to the one used in \cite{fleming1989existence} in the context of differential games.

\vs2

We first construct the sequence of discrete time optimal control problems. For $n\ge 1$,  let $\pi_{n} := \{t_{j}^{n}, j \leq 2^{n}\}$ with $t_{j}^{n} := jT/2^{n}$, and let $\Phi_{n}^{z,m}$ be the set of controls $\phi=(\tau_{i}^{\phi},\alpha_{i}^{\phi})_{i\ge 1}$ in $\Phi^{z,m}$ such that $(\tau_{i}^{\phi})_{i\ge 1}$ takes values in $\pi_{n}\cup\{t\}\cup [T,\infty)$, if $z=(t,x)$. The corresponding value function is 
\[
    \vr_{n}(z, m) = \sup_{\phi \in \Phi^{z,m}_{n}}J(z,m,\phi), \; \; (z,m)\in \Zb\x \Mb.
\]
We extend $\vr_{n}$ by setting 
\begin{align}\label{eq: extension vrn}
 \vr_{n} :=\Kc_{T}g , \; \;\mbox{ on } (T,\infty)\x \Mb,
 \end{align}

\begin{Remark}\label{rem: vrn le vr} Note that $\vr_{n} \le \vr\le \vr^{*}$ by construction.  
\end{Remark}

We first prove that $\vr_{n}$ satisfies a dynamic programming principle. This requires additional notations. We first define the next time on the grid at which a new action can be made, given that $a$ is plaid: 
$$
s^{n,a}[t,x]:=\min\{s\in \pi_{n}\cup[T,\infty): s\ge  \varpi(t,x,a, \upsilon, \epsilon_{j})\mbox{ and } s>t\} .
$$
Let $\partial$ denote a cemetery point that does not belong to $\Ab$. Given $a\in \Ab\cup\{\partial\}$, we make a slight abuse of notation by denoting by  $(Z^{(t,x),a},M^{(t,x),m,a})$ the process defined as   $(Z^{(t,x),\phi},M^{(t,x),m,\phi})$ for $\phi$ such that 
$$
(\tau^{\phi}_{1},\alpha^{\phi}_{1})= (t,a)\1_{\{a\ne \partial\}}+(T+1,a_{\star})\1_{\{a= \partial\}}
$$ 
in which $a_{\star}\in \Ab$ and $\tau^{\phi}_{i}>T+ 1$ for $i>1$. Then, we  set 
\begin{align*}
\bar J(T,\cdot;a):=\Kc_{T}\Kc^{a} g \;,\;\bar \vr_{n}(T,\cdot):=\sup_{a\in \Ab\cup \{\partial\}}\bar J(T,\cdot;a) \; \mbox{ on } \R^{d}\x \Mb\x (\Ab\cup \{\partial\}),\;\;
 \end{align*}
 with the convention that $\Kc^{\partial}$ is the identity, 
and define by backward induction on the intervals $[t^{n}_{j},T)$, $j=n-1,\cdots,0$, 
\begin{align*}
\bar J(z,m;a)&:=\E_{m}[\bar \vr_{n}(Z^{z,a}_{s^{n,a}[z]},M^{z,m,a}_{s^{n,a}[z]} )]
\;,\;
\bar \vr_{n}:= \sup_{a\in \Ab\cup \{\partial\}}\bar J(\cdot;a),
\end{align*}
together with the extension
$$
\bar \vr_{n}:=\Kc_{T}g \;\mbox{ on } (T,\infty)\x\R^{d}\x \Mb.
$$

\begin{Lemma}\label{lem : recu backward pour dpp} Fix $\iota>0$. Then, there  exists a universally measurable map $(z,m)\in \Zb\x \Mb \mapsto \hat a^{n,\iota}[z,m]\in \Ab\cup\{\partial\}$ such that 
$
\bar J(\cdot;\hat a^{n,\iota}[\cdot])\ge \bar \vr_{n}-\iota$ on $\Zb\x \Mb.
$
Moreover, the map $\bar \vr_{n}$ is upper semi-analytic.
\end{Lemma}

\proof Since $\Kc_{T}g$ is assumed to be upper semi-analytic (indeed continuous), it follows from \cite[Proposition 7.48 page 180]{dimitri1996stochastic} that $\bar J$ is upper semi-analytic on $[t_{n-1}^{n},T]\x \R^{d}\x \Mb \x (\Ab\cup\{\partial\})$. Then, the required result  holds on $[t_{n-2}^{n},T]\x \R^{d}\x \Mb$ by \cite[Proposition 7.50 page 184]{dimitri1996stochastic}. It is then extended to $[0,T]\x \R^{d}\x \Mb$ by a backward induction. \ep

\begin{Proposition}\label{prop: bar vn = vn} $\bar \vr_{n}=\vr_{n}$ on $\Zb\x \Mb$. Moreover, given a random variable $(\zeta,\mu)$ with values in $\Zb\x \Mb$ and $\iota>0$, there exists a measurable map  $(z,m)\mapsto \phi^{\iota}[z,m]$ such that 
$$
  J(\zeta,\mu; \phi^{\iota}[\zeta,\mu])\ge   \vr_{n}(\zeta,\mu)-\iota~~\P_{m}-{\rm a.s.}
$$

\end{Proposition}

\proof The proof proceeds by induction.    Our claim follows from definitions on $[t^{n}_{n},T]\x\R^{d}\x\Mb$. Assume that it holds on $[ t^{n}_{j+1}, T]\x\R^{d}\x\Mb$ for some $j\le n-1$. For the following, we fix  $z=(t,x)\in \Zb$ with $t\in [t^{n}_{j},t^{n}_{j+1})$ and $m\in \Mb$.

{Step 1:} In this step, we first construct a suitable candidate to be an almost-optimal control. Fix $\eps_{1},\ldots,\eps_{n}>0$, $\eps_{0}:=0$,  and set $\eps(i):=(\eps_{0},\eps_{1},\ldots,\eps_{i})$. Let $(\hat a^{n,\iota})_{\iota>0}$ be  as in Lemma \ref{lem : recu backward pour dpp}, and consider its extension defined by $\hat a^{n,\iota}=a_{\star}$ on $(T,\infty)\x\R^{d}\x \Mb$. Define $r^{\eps{(0)}}_{1}:=t$ and 
$\phi^{\eps{(1)}}_{1}\in \Phi^{z,m}_{n}$ by 
$$
(\tau^{\phi^{\eps{(1)}}_{1}}_{i}, \alpha^{\phi^{\eps{(1)}}_{1}}_{i})= (r^{\eps{(0)}}_{1}, \tilde a^{n,\eps_{1}}[r^{\eps{(0)}}_{1},x,m])\1_{\{i=1\}}+\1_{\{i>1\}}(T+i,a_{\star})  \;,\;\;i\ge 1.   
$$
where 
$$
\tilde a^{n,\eps_{1}}[r^{\eps{(0)}}_{1},x,m]:=\hat a^{n,\eps_{1}}[r^{\eps{(0)}}_{1},x,m] . 
$$
We then set 
\begin{align*}
r^{\eps{(1)}}_{2}&:=\min  \pi_{n}\cap[ \vartheta^{\phi^{\eps{(1)}}_{1}}_{1},2T]\cap (r^{\eps{(0)}}_{1},\infty).
\end{align*}
By  Lemma \ref{lem : recu backward pour dpp} and  \cite[Lemma 7.27 page 173]{dimitri1996stochastic} applied to the pull-back measure of $(Z^{z,\phi^{\eps{(1)}}_{1}}_{r^{\eps(1)}_{2}},$ $M^{z,m,\phi^{\eps{(1)}}_{1}}_{r^{\eps(1)}_{2}})$, we can find a Borel measurable map  $(t',x',m')\in \Zb\x \Mb \mapsto \tilde a^{n,\eps_{2}}_{2}[t',x',m']\in \Ab\cup\{\partial\}$ such that 
$$ 
\tilde a^{n,\eps_{2}}[Z^{z,\phi^{\eps{(1)}}_{1}}_{r^{\eps(1)}_{2}},M^{z,m,\phi^{\eps{(1)}}_{1}}_{r^{\eps(1)}_{2}}]= \hat a^{n,\eps_{2}}[Z^{z,\phi^{\eps{(1)}}_{1}}_{r^{\eps{(1)}}_{2}},M^{z,m,\phi^{\eps{(1)}}_{1}}_{r^{\eps{(1)}}_{2}}]\;\;\;\P_{m}-\as
$$
We define  $\phi^{\eps(2)}_{2}$ by 
\begin{align*}
&(\tau^{\phi^{\eps(2)}_{2}}_{i}, \alpha^{\phi^{\eps(2)}_{2}}_{i})=(r^{\eps(1)}_{2}, \tilde a^{n,\eps_{2}}[Z^{z,\phi^{\eps(1)}_{1}}_{r^{\eps(1)}_{2}},M^{z,m,\phi^{\eps(1)}_{1}}_{r^{\eps(1)}_{2}}])\1_{\{i=2,r^{\eps(1)}_{2}\le T\}} +(\tau^{\phi^{\eps(1)}_{1}}_{i}, \alpha^{\phi^{\eps(1)}_{1}}_{i})\1_{\{i\ne 2\}\cup \{r^{\eps(1)}_{2}> T\}},
\end{align*}
for $i\ge 1$. 
We then define recursively for $k\ge 2$  
\begin{align*}
r^{\eps{(k)}}_{k+1}:=&\inf \pi_{n}\cap[ \vartheta^{\phi^{\eps{(k)}}_{k}}_{k},2T]\cap (r^{\eps(k-1)}_{k},\infty)\\
(\tau^{\phi^{\eps(k+1)}_{k+1}}_{i}, \alpha^{\phi^{\eps(k+1)}_{k+1}}_{i})=&(r^{\eps(k)}_{k+1}, \tilde a^{n,\eps_{k+1}}[Z^{z,\phi^{\eps(k)}_{k}}_{r^{\eps(k)}_{k+1}},M^{z,m,\phi^{\eps(k)}_{k}}_{r^{\eps(k)}_{k+1}}])\1_{\{i=k+1,r^{\eps(k)}_{k+1}\le T\}}\\
& +(\tau^{\phi^{\eps(k)}_{k}}_{i}, \alpha^{\phi^{\eps(k)}_{k}}_{i})\1_{\{i\ne k+1\}\cup \{r^{\eps(k)}_{k+1}> T\}},   
\end{align*}
for $i\ge 1$, 
in which  $(t',x',m')\in \Zb\x \Mb \mapsto \tilde a^{n,\eps_{k+1}}_{k+1}[t',x',m']\in \Ab\cup\{\partial\}$ is a Borel measurable map such that 
$$ 
\tilde a^{n,\eps_{k+1}}[Z^{z,\phi^{\eps(k)}_{k}}_{r^{\eps(k)}_{k+1}},M^{z,m,\phi^{\eps(k)}_{k}}_{r^{\eps(k)}_{k+1}}]= \hat a^{n,\eps_{k+1}}[Z^{z,\phi^{\eps(k)}_{k}}_{r^{\eps(k)}_{k+1}},M^{z,m,\phi^{\eps(k)}_{k}}_{r^{\eps(k)}_{k+1}}]\;\;\;\P_{m}-\as
$$
We finally set 
$$
\phi^{\eps}:= (\tau^{\phi^{\eps(i)}_{i}}_{i}, \alpha^{\phi^{\eps(i)}_{i}}_{i})_{i\ge 1}\in \Phi^{z,m}_{n}.
$$  
 
Step 2: We now prove that $\bar \vr_{n}(z,m)\ge  \vr_{n}(z,m)$. By the above construction and Lemma \ref{lem : recu backward pour dpp}, 
\begin{align*}
 \bar \vr_{n}(z,m)\ge\bar J(z,m; \alpha^{\phi^{\eps(1)}_{1}}_{1})&\ge   \bar \vr_{n}(z,m)-\eps_{1}. 
\end{align*}
Since $\vr_{n}(t_{k},\cdot)=\bar \vr_{n}(t_{k},\cdot)$ for $k>j$ by our induction hypothesis, we obtain 
\begin{align*}
\bar \vr_{n}(z,m)& \ge   \sup_{a\in \Ab\cup\{\partial\}} \E_{m}[ \vr_{n}(Z^{z,a}_{r^{\eps(1)}_{2}},M^{z,m,a}_{r^{\eps(1)}_{2}} )] -\eps_{1}\ge \vr_{n}(z,m)-\eps_{1},
\end{align*}
in which the last inequality follows from a simple conditioning argument as in the proof of Proposition \ref{prop: DPP}. 
By arbitrariness of $\eps_{1}>0$, this implies that $\bar \vr_{n}(z,m)\ge  \vr_{n}(z,m)$. 

Step 3: It remains  to prove that $\bar \vr_{n}(z,m)\le  \vr_{n}(z,m)$. Define 
$$
Y_{i}^{\eps(i-1)}:=(Z^{z,\phi^{\eps}}_{r^{\eps(i-1)}_{i}},M^{z,m,\phi^{\eps}}_{r^{\eps(i-1)}_{i}}),\;i\ge 1,
$$
with $Y_{0}^{\eps(-1)}:=(z,m)$, and observe that $Y_{i}^{\eps(i-1)}$ and $\Fc_{r^{\eps(i-1)}_{i}}^{z,m,\phi^{\eps}}$ only depend on $\eps(i-1)$. Then, 
for each $i\ge 0$, 
\begin{align*}
\bar \vr_{n}(Y_{i}^{\eps(i-1)})&=\lim_{\eps_{i}\downarrow 0} \E_{m}[\bar \vr_{n}(Z^{Y_{i}^{\eps(i-1)},\phi^{\eps(i)}_{i}}_{r^{\eps(i)}_{i+1}},M^{Y_{i}^{\eps(i-1)},\phi^{\eps(i)}_{i}}_{r^{\eps(i)}_{i+1}})|\Fc_{r^{\eps(i-1)}_{i}}^{z,m,\phi^{\eps}}]] \
\\
 &=\lim_{\eps_{i}\downarrow 0} \E_{m}[\1_{\{r^{\eps(i)}_{i+1}\le T\}}\bar \vr_{n}(Z^{Y_{i}^{\eps(i-1)},\phi^{\eps(i)}_{i}}_{r^{\eps(i)}_{i+1}},M^{Y_{i}^{\eps(i-1)},\phi^{\eps(i)}_{i}}_{r^{\eps(i)}_{i+1}})|\Fc_{r^{\eps(i-1)}_{i}}^{z,m,\phi^{\eps}}]\\
 &+\lim_{\eps_{i}\downarrow 0}\E_{m}[\1_{\{r^{\eps(i)}_{i+1}> T\}}g(Z^{Y_{i}^{\eps(i-1)},\phi^{\eps(i)}_{i}}_{r^{\eps(i)}_{i+1}},M^{Y_{i}^{\eps(i-1)},\phi^{\eps(i)}_{i}}_{r^{\eps(i)}_{i+1}},\upsilon,\epsilon_{0})|\Fc_{r^{\eps(i-1)}_{i}}^{z,m,\phi^{\eps}}]\;\;\;\;\;\P_{m}-\as 
\end{align*}
on $\{r^{\eps(i-1)}_{i}\le T\}$.
Since $g$ is bounded, so is $\bar \vr_{n}$. The above combined with the dominated convergence theorem then implies 
\begin{align*}
\bar \vr_{n}(z,m)&=\lim_{\eps_{1}\downarrow 0}\cdots\lim_{\eps_{n}\downarrow 0} \E_{m}[\sum_{i=0}^{n} \1_{\{r^{\eps(i)}_{i+1}> T\ge r^{\eps(i-1)}_{i} \}}g(Z^{Y_{i}^{\eps(i-1)},\phi^{\eps(i)}_{i}}_{r^{\eps(i)}_{i+1}},M^{Y_{i}^{\eps(i-1)},\phi^{\eps(i)}_{i}}_{r^{\eps(i)}_{i+1}},\upsilon,\epsilon_{0})] 
\\
&=\lim_{\eps_{1}\downarrow 0}\cdots\lim_{\eps_{n}\downarrow 0} J(z,m;\phi^{\eps}) 
\le \vr_{n}(z,m),
\end{align*}
which concludes the proof that $\bar \vr_{n}=\vr_{n}$. 

Step 4. The second assertion of the proposition is obtained by observing that, given a random variable $(\zeta,\mu)$ with values in $\Zb\x \Mb$,  one can choose  $\tilde a^{n,\eps_{1}}$ Borel measurable such that 
$
\tilde a^{n,\eps_{1}}[\zeta,\mu]=\hat a^{n,\eps_{1}}[\zeta,\mu] $ $\P_{m}-{\rm a.s.}
$
\ep 
\vs2
 
We are now in position to conclude that $\vr_{n}$ satisfies a dynamic programming principle. 
\begin{Corollary}\label{cor: dpp discret} Fix $z=(t,x)\in \Zb$ and $m\in \Mb$. Let $(\theta^{\phi},\phi\in \Phi^{z,m}_{n})$ be such that each $\theta^{\phi}$ is a   $\F^{z,m,\phi}$-stopping time with values in $[t, 2T]\cap (\pi_{n}\cup [T,\infty))$ such that 
$
\theta^{\phi}\in \Nc^{\phi}\cap [t,\T[\phi]]$ $\P_{m}-\as  
$
for $\phi \in \Phi^{z,m}_{n}$.
Then,   
 $$
  \vr_{n}(z,m)= \sup_{\phi\in \Phi^{z,m}_{n}} \E_{m}[  \vr_{n}(Z^{z,\phi}_{\theta^{\phi}},M^{z,m,\phi}_{\theta^{\phi}} )]. 
$$ 
 
\end{Corollary}

\proof  The inequality $\le$ can be obtained trivially by a conditioning argument.  Fix $\phi \in  \Phi^{z,m}_{n}$. By Proposition \ref{prop: bar vn = vn}, we can find a Borel measurable map  $(z',m')\mapsto \phi^{\iota}[z',m']$ such that 
$$
J(Z^{z, \phi}_{\theta^{\phi}},M^{z,m,  \phi}_{\theta^{\phi}};\phi^{\iota}[Z^{z, \phi}_{\theta^{\phi}},M^{z,m,  \phi}_{\theta^{\phi}}] )\ge   \vr_{n}(Z^{z, \phi}_{\theta^{\phi}},M^{z,m,  \phi}_{\theta^{\phi}} )-\iota.
  $$
  Let us now simply write $\phi^{\iota}$ for $\phi^{\iota}[Z^{z, \phi}_{\theta^{\phi}},M^{z,m,  \phi}_{\theta^{\phi}}] $. Without loss of generality, one can assume that $\tau^{\phi}_{1}\ge t$ and that $\tau^{\phi^{\iota}}_{1}\ge \theta^{\phi}$. Let $I:={\rm card}\{i\ge 1: \tau^{\phi}_{i}<\theta^{\phi}\}$.
  Then, 
  $
   J(z,m;\tilde \phi^{\iota})
 \ge \E_{m}[    \vr_{n}(Z^{z, \phi}_{\theta^{\phi}},M^{z,m,  \phi}_{\theta^{\phi}} )]-\iota
 $
 in which  
$
 (\tau^{\tilde \phi^{\iota}}_{i},\alpha^{\tilde \phi^{\iota}}_{i})= \1_{i\le I}    (\tau^{\phi}_{i},\alpha^{\phi}_{i})+\1_{i>I} (\tau^{\phi^{\iota}}_{i-I},\alpha^{\phi^{\iota}}_{i-I}),$ $i\ge 1$.
Sending $\iota \to 0$ leads to the required result.
\ep

\subsection{Super-solution property as the time step vanishes}\label{sec: super sol vrinfty}

We now consider the   limit $n\to \infty$. Let us set,  for $(z,m)\in \R_{+}\x \R^{d}\x\Mb$, 
    \[
        \begin{aligned}
              \vr_{\circ}(z,m) := \liminf_{(t',x', m',n) \rightarrow (z, m,\infty)}\vr_{n}(t',x', m').
        \end{aligned}
    \]

\begin{Remark}\label{rem: vrcic=KcTg apres T} Note that \reff{eq: extension vrn} and \reff{eq: hyp conti KcT et Kc} implies that $\vr_{\circ}=\Kc_{T}g$ on $(T,\infty)\x \R^{d}\x \Mb$.
\end{Remark}

\begin{Proposition}\label{prop : vrcirc sursol} The function $ \vr_{\circ}$ is a viscosity super-solution of \reff{eq: pde interior}-\reff{eq: pde T}.
\end{Proposition}  

\proof Let $n_{k}\to \infty$ and $(z_{k},m_{k})\to (z_{o},m_{\circ})$ be such that   $\vr_{n_{k}}(z_{k},m_{k})\to \vr_{\circ}(z_{o},m_{o})$. 
\\
Step 1. We first show that $\vr_{\circ} (z_{\circ},m_{\circ})\ge \Kc \vr_{\circ}(z_{\circ},m_{\circ})$.  By Corollary \ref{cor: dpp discret} applied to $\vr_{n_{k}}$ with  a control $\phi^{k}$ defined by 
$
(\tau^{k}_{i},\alpha^{k}_{i})=(t_{k},a_{k})\1_{\{i=1\}} + \sum_{j>1} (T+j,a_{\star})\1_{\{i=j\}},$ $i\ge 1,
$
with $a_{k}\in   \Ab$, 
we obtain 
\begin{align*}
\vr_{n_{k}}(z_{k},m_{k})&\ge \sup_{a_{k}\in \Ab}\int \E[\vr_{n_{k}}(Z^{z',\circ}_{s^{n_{k}}_{+}[z']},m')]d\kr(z',m'|z_{k},m_{k},a_{k})]
= \Kc \E[\vr_{n_{k}}(Z^{\cdot,\circ}_{s^{n_{k}}_{+}[\cdot]},\cdot)](z_{k},m_{k}), 
\end{align*}
in which $s^{n_{k}}_{+}[t,x]:=\min  \pi_{n_{k}}  \cap [t,\infty)$.  
Let $\vp_{k_{\circ}}$ be the lower-semicontinuous enveloppe of $\inf\{ \E[\vr_{n_{k}}(Z^{\cdot,\circ}_{s^{n_{k}}_{+}[\cdot]},\cdot)],k\ge k_{\circ}\}$.  
Then,  for $k\ge k_{\circ}$,
$
\vr_{n_{k}}(z_{k},m_{k})\ge \int \vp_{k_{\circ}}(z',m')d\kr(z',m'|z_{k},m_{k},a_{k}),
$
and, by \reff{eq: hyp conti KcT et Kc}, passing to the limit $k\to \infty$ leads to 
$
\vr_{\circ}(z_{\circ},m_{\circ})\ge \int \vp_{k_{\circ}}(z',m')d\kr(z',m'|z_{\circ},m_{\circ},a_{\circ}).
$
We shall prove in step 3 that $\lim_{k_{\circ}\to \infty}\vp_{k_{\circ}}\ge \vr_{\circ}$. These maps are  bounded, since $g$ is.  Dominated convergence then implies that 
$\vr_{\circ}(z_{\circ},m_{\circ})\ge \int \vr_{\circ}(z',m')d\kr(z',m'|z_{\circ},m_{\circ},a_{\circ}).
$

Step 2. Let $\vp$ be a (bounded) $C^{1,2,0}([0,T]\x \R^{d}\x \Mb)$ function and $(z_{\circ},m_{\circ})\in [0,T)\x \R^{d}\x \Mb$ be a minimal point of $\vr_{\circ}-\vp$ on $\Zb\x \Mb$. Without loss of generality, one can assume that $(\vr_{\circ}-\vp)(z_{\circ},m_{\circ})$ $=$ $0$.  
 Let $B$  and   $(z_{k},m_{k},n_{k})_{n\ge 1}$ be as in Lemma \ref{lem: approx stabi} below. We write $z_{k}=(t_{k},x_{k}) , z_{\circ}=(t_{\circ},x_{\circ})\in [0,T]\x \R^{d}$. 
On the other hand, by considering the control 
$\phi^{k}$ defined by 
$
(\tau^{k}_{i},\alpha^{k}_{i})= (T+i,a_{\star}),\;i\ge 1,
$
we obtain from Corollary \ref{cor: dpp discret} that 
\begin{align*}
\vr_{n_{k}}(z_{k},m_{k})&\ge \E_{m}[ \vr_{n_{k}}(Z^{z_{k},\circ}_{t_{k}+  h_{k}},m)] \;\;
\end{align*}
 with $h_{k}\in  T2^{-n_{k}}(\N\cup \{0\})$ such that  $t_{k}+ h_{k}<T$ if $t_{\circ}\ne T$ and $t_{k}+ h_{k}=T$ otherwise.

 Let $C>0$ be a common bound for  $(\vr_{n})_{n\ge 1}$ and $\vp$.  Then  we can choose $(h_{k})_{k\ge 1}$ such that 
$$
 \delta_{k}:=(\vp(z_{k},m_{k})-\vr_{n_{k}}(z_{k},m_{k})-2C\;\P[Z^{z_{k},\circ}_{t_{k}+ h_{k}}\notin B])/h_{k}\to 0.
$$
This follows from standard estimates on the solution of sde's with Lipschitz coefficients. 
Then, if $t_{\circ}<T$, 
\begin{align*}
0&\ge  h_{k}^{-1}\E_{m}[ \vp(Z^{z_{k},\circ}_{t_{k}+ h_{k}},m_{k})-\vp_{n_{k}}(z_{k},m_{k})]+ \delta_{k}= \E_{m}[ h_{k}^{-1}\int_{t_{k}}^{t_{k}+h_{k}}\Lc\vp(Z^{z_{k},\circ}_{s},m_{k})ds]+ \delta_{k},
\end{align*}
sending $k\to \infty$ leads to $\Lc\vp(z_{\circ},m_{\circ})\le 0$. If $t_{\circ}=T$,   
$
\vr_{n_{k}}(z_{k},m_{k})\ge \E_{m}[ g(Z^{z_{k},\circ}_{T},m_{k},\upsilon,\epsilon_{0})] =\E_{m}[ \Kc_{T}g(Z^{z_{k},\circ}_{T},m_{k})]
$
and passing to the limit leads to 
$
\vp(z_{\circ},m_{\circ})\ge  \Kc_{T}g(z_{\circ},m_{\circ}),
$
recall \reff{eq: hyp conti KcT et Kc}. Finally, $\vp (z_{\circ},m_{\circ})\ge \Kc \vp(z_{\circ},m_{\circ})$ by Step 1.\ep

Step 3: It remains to prove the claim used in Step 1. Let us set 
	\[
		\bar \varphi_{k_{\circ}}(z', m') :=  \inf_{k \geq k_{\circ}}\left\{\mathbb{E}\left[\vr_{n_{k}}\left(Z_{s_{+}^{n_{k}}}^{z', \circ}[z'], m')\right)\right]\right\},
	\]
	so that $\varphi_{k_{\circ}}$ is the lower-semicontinuous envelope of $\bar \varphi_{k_{\circ}}$. Note that  $Z_{s_{+}^{n_{k}}[z']}^{z', \circ} $ converges a.s.~to  $z$ as $(z',k)\to (z,\infty)$. Hence, for all $\eps>0$, there exist open neighborhoods $B_{\eps}(z,m)$ and $B_{\frac{\eps}{2}}(z,m)$ of $(z,m)$, as well as $k_{\eps}\in \N$ such that $\P[(Z_{s_{+}^{n_{k}}[z']}^{z', \circ} ,m')\notin B_{\eps}(z,m)]\le \eps$ for $k\ge k_{\eps}$ and $(z',m')\in B_{\frac{\eps}{2}}(z,m)$. One can also choose $k_{\eps}$ and $B_{\frac{\eps}{2}}(z,m)$ such that  
	$
		\inf_{k \geq k_{\varepsilon}}\vr_{n_{k}}(z', m') \geq \vr_{\circ}(z, m') - \varepsilon
	$
	 for all $k \geq k_{\varepsilon}$ and $(z',m') \in B_{\frac{\eps}{2}}(z,m)$. Let $C>0$ be a bound for $(|\vr_{n}|)_{n\ge 1}$ and $|\vr_{\circ}|$, recall that $g$ is bounded. 
Then, for $k_{\circ}$ large enough and $(z',m') \in B_{\frac{\eps}{2}}(z,m)$, 
		\begin{align*}
			\bar \varphi_{k_{\circ}}(z', m') \geq   \vr_{\circ}(z, m) -\eps-2C \sup_{k\ge k_{\circ}}\P[(Z_{s_{+}^{n_{k}}[z']}^{z', \circ},m') \notin B_{\eps}(z,m)]
			\ge  \vr_{\circ}(z, m)-\eps(1+2C). 
		\end{align*}
Hence, since  $ \vr_{\circ}$ is lower-semicontinuous, 
$$
\lim_{k_{\circ}\to \infty}\varphi_{k_{\circ}}(z, m)=\lim_{k_{\circ}\to \infty}\liminf_{(z',m')\to (z,m)} \bar \varphi_{k_{\circ}}(z', m') \geq   \vr_{\circ}(z, m). 
$$
 \ep
\vs2

We conclude this section with the technical lemma that was used in the above proof. 
 
 \begin{Lemma}\label{lem: approx stabi}  Let   $(u_{n})_{n\ge 1}$ be a sequence of  lower semi-continuous maps on  $\Zb\x \Mb$ and define 
 $u_{\circ}:=\liminf_{(z',m',n)\to (\cdot,\infty)} u_{n}(z',m')$ on $ \Zb\x \Mb$. Assume that $u_{\circ}$ is locally bounded.  Let $\vp$ be a  continuous map and assume that $(z_{\circ},m_{\circ})$ is a strict minimal point of $u_{\circ}-\vp$ on $\Zb\x \Mb$. Then, one can find   a bounded open set  $B$  of $[0,T]\x \R^{d}$ and  a sequence $(z_{k},m_{k},n_{k})_{n\ge 1}\subset B\x\Mb\x\N$ such that $n_{k}\to \infty$, $(z_{k},m_{k})$ is a   minimum point of  $u_{n_{k}}-\vp$ on $ B \x \Mb$ and   $(z_{k},m_{k},u_{n_{k}}(z_{k},m_{k}))\to (z_{o},m_{\circ},u_{\circ}(z_{o},m_{o}))$.
 \end{Lemma} 

\proof Since $\Mb$ is assumed to be locally compact, it suffices  to repeat the arguments in the proof of \cite[p80, Proof of Lemma 6.1]{barles1994solutions}.
\ep

\subsection{Conclusion of the proof of Theorem \ref{thm:viscosity}}\label{sec: conc proof super sol}  

 We already know from Proposition \ref{prop: subsol} and Proposition \ref{prop : vrcirc sursol} that $\vr^{*}$ and $\vr_{\circ}$ are respectively    a bounded viscosity sub- and super-solution of \reff{eq: pde interior}-\reff{eq: pde T}. By \reff{eq: def vr}, Remark \ref{rem: vrcic=KcTg apres T} and \reff{eq: hyp conti KcT et Kc}, we also have $\vr_{\circ}\ge \vr^{*}$ on $(0,T)\x \R^{d}\x \Mb$.  In view of Assumption \ref{ass: comp} and Remark \ref{rem: vrn le vr},  $\vr$ is continuous on $\Zb\x \Mb$ and is the unique bounded viscosity solution of \reff{eq: pde interior}-\reff{eq: pde T}.   \ep 
 
 \begin{Remark}{The above arguments actually show that $(\vr_{n})_{n\ge 1}$ converges to $\vr$. }
 \end{Remark}
 
 \section{A sufficient condition for the comparison } \label{preuve_comparaison}
 
 In this section, we provide a sufficient condition for Assumption \ref{ass: comp} to hold. {We refer to \cite{BBD16extended} for examples of application.}

 \begin{Proposition}\label{comparaison}  Assumption \ref{ass: comp}  holds whenever  there exists a continuous function $\Psi$ on $[0,2T]\x \R^{d} \times \Mb$ satisfying
\begin{enumerate}[{\rm (i)}]
\item $\Psi(., m) \in C^{1,2}([0,T)\x \R^{d})$, for all $m\in \Mb$.
\item $\varrho \Psi\ge \Lc\Psi$ on $[0,T]\x\R^{d}\x \Mb$ for some constant $\varrho > 0$,
\item $\Psi-\Kc\Psi\geq \delta$ on $[0,T]\x\R^{d}\x \Mb$ for some $\delta > 0$,
\item $\Psi \geq \Kc_{T}[\tilde{g} ]$ on $[T,\infty)\x \R^{d}\x \Mb$ with $\tilde{g}(t, .) := e^{\varrho t}g(t, .)$ and $\varrho$ is defined in \emph{(ii)},
\item $\Psi^{-}$ is bounded.
\end{enumerate}
 \end{Proposition}	
The idea of the proof is the same as in \cite[Proposition 4.12]{bouchard2009stochastic}. Note that their condition {\rm H2 (v)} is not required here because we only consider bounded sub and super-solutions and we take a different approach. To avoid it, we slighlty reinforce the hypothesis {\rm H2 (iii)} and asked for $\Psi^{-}$ to be bounded.
\vs2 

\proof
Step 1. As usual, we shall argue by contradiction. We assume that there exists $(z_{0}, m_{0}) \in \mathbf{Z} \times \mathbf{M}$ such that $(U-V)(z_{0}, m_{0}) > 0$, in which  $U$ and $V$ are as in Assumption \ref{ass: comp}. Recall the definition of $\Psi$, $\varrho$ and $\tilde g$ in Proposition \ref{comparaison}. We set $\tilde{u}(t, x, m) := e^{\varrho t}U(t, x, m)$ and $\tilde{v}(t, x, m) := e^{\varrho t}V(t, x, m)$ for all $(t, x, m) \in \mathbf{Z} \times \Mb$. Then, there exists $\lambda\in (0,1)$  such that

\begin{equation}\label{comparaison_0}
		(\tilde{u}-\tilde{v}^{\lambda})(z_{0}, m_{0}) > 0,
\end{equation}
in which $\tilde{v}^{\lambda} := (1-\lambda)\tilde{v} + \lambda\Psi$. Note that $\tilde{u}$ and $\tilde{v}$ are sub and supersolution on $\mathbf{Z} \times \Mb$ of
	\begin{equation}
		\min\left\{\varrho\varphi - \mathcal{L}\varphi, \varphi - \mathcal{K}\varphi\right\} = 0
	\end{equation}
associated to the boundary condition
	\begin{equation}\label{comparaison_bord}
		\min\left\{\vp- \Kc_{T}\tilde{g},\vp- \Kc \vp\right\}=0.
	\end{equation}

Step 2. Let ${\rm d}_{\mathbf{M}}$ be a metric on $\Mb$ compatible with the topology of weak convergence. For $(t, x, y, m) \in \mathbf{Z} \times \Xb \times \Mb$, we set
	\begin{equation}\label{gamma_vareps}
		\Gamma_{\varepsilon}(t, x, y, m) := \tilde{u}(t, x, m)-\tilde{v}^{\lambda}(t, y, m) - \varepsilon\left(\|x\|^{2} + \|y\|^{2} + {\rm d}_{\mathbf{M}}(m)\right)
	\end{equation}
with $\varepsilon>0$ small enough such that $\Gamma_{\varepsilon}(t_{0}, x_{0}, x_{0}, m_{0}) > 0$. Note that the supremum of $(t, x, m) \mapsto \Gamma_{\varepsilon}(t, x, x, m)$ over $\mathbf{Z} \times \Xb\x \mathbf{M}$ is achieved by some $(t_{\varepsilon}, x_{\varepsilon}, x_{\varepsilon}, m_{\varepsilon})$. This follows from the  the upper semi-continuity of $\Gamma_{\varepsilon}$ and the fact that $\tilde{u}, -\tilde{v}, -\Psi$ are bounded from above. Recall that  $\mathbf{M}$ is locally compact. 
For $(t, x, y, m) \in \mathbf{Z} \times \mathbf{X} \times \mathbf{M}$, we set
	\begin{equation*}
		\Theta_{\varepsilon}^{n}(t, x, y, m) := \Gamma_{\varepsilon}(t, x, y, m) - n\|x - y\|^{2}.
	\end{equation*}
Again, there is $(t_{n}^\eps, x_{n}^\eps, y_{n}^\eps, m_{n}^\eps) \in \mathbf{Z} \times \mathbf{X} \times \mathbf{M}$ such that
	$
		\sup_{\mathbf{Z} \times \mathbf{X} \times \mathbf{M}}\Theta_{\varepsilon}^{n} = \Theta_{\varepsilon}^{n}(t_{n}^\eps, x_{n}^\eps, y_{n}^\eps, m_{n}^\eps).
	$
It is  standard to show that, after possibly considering a subsequence, 
	\begin{equation}\label{comparaison_standard}
		\begin{aligned}
			&(t_{n}^\eps, x_{n}^\eps, y_{n}^\eps, m_{n}^\eps) \rightarrow (\hat{t}_{\varepsilon}, \hat{x}_{\varepsilon}, \hat{x}_{\varepsilon}, \hat{m}_{\varepsilon})\in \Zb\x \Xb\x \Mb,   \ \ n\|x_{n}^\eps-y_{n}^\eps\|^{2} \rightarrow 0,  \\
			&\text{ and}\;\;\Theta_{\varepsilon}^{n}(t_{n}^\eps, x_{n}^\eps, y_{n}^\eps, m_{n}^\eps) \rightarrow \Gamma_{\varepsilon}(\hat{t}_{\varepsilon}, \hat{x}_{\varepsilon}, \hat{x}_{\varepsilon}, \hat{m}_{\varepsilon}) = \Gamma_{\varepsilon}(t_{\varepsilon}, x_{\varepsilon}, x_{\varepsilon}, m_{\varepsilon}),
		\end{aligned}
	\end{equation}
see e.g.~\cite[Lemma 3.1]{crandall1992user}.

Step 3. We first assume that, up to a subsequence,
	$
		(\tilde{u} - \mathcal{K}\tilde{u})(t_{n}^\eps, x_{n}^\eps, m_{n}^\eps) \leq 0,$ for $ n \geq 1.
	$
It follows from the supersolution property of $\tilde{v}$ and Condition (iii) of Proposition \ref{comparaison} that
	\[
		\tilde{u}(t_{n}^\eps, x_{n}^\eps, m_{n}^\eps) - \tilde{v}^{\lambda}(t_{n}^\eps, y_{n}^\eps, m_{n}^\eps) \leq \mathcal{K}\tilde{u}(t_{n}^\eps, x_{n}^\eps, m_{n}^\eps) -\mathcal{K} \tilde{v}^{\lambda}(t_{n}^\eps, y_{n}^\eps, m_{n}^\eps) - \lambda \delta.
	\]
Passing to the $\limsup$ and using \reff{comparaison_standard} and \reff{eq: hyp conti KcT et Kc}, we obtain 
	$
		(\tilde{u} - \tilde{v}^{\lambda})(\hat{t}_{\varepsilon}, \hat{x}_{\varepsilon}, \hat{m}_{\varepsilon}) + \lambda \delta \leq \mathcal{K}(\tilde{u} - \tilde{v}^{\lambda}) (\hat{t}_{\varepsilon}, \hat{x}_{\varepsilon}, \hat{m}_{\varepsilon}).
	$
In particular, by (\ref{gamma_vareps}),
	$
		\Gamma_{\varepsilon}(\hat{t}_{\varepsilon}, \hat{x}_{\varepsilon}, \hat{x}_{\varepsilon}, \hat{m}_{\varepsilon}) + \lambda \delta \leq \mathcal{K}(\tilde{u} - \tilde{v}^{\lambda})(\hat{t}_{\varepsilon}, \hat{x}_{\varepsilon}, \hat{m}_{\varepsilon}).
	$
Now let us  observe that 
	\begin{align}
		\sup_{\mathbf{Z} \times \mathbf{M}} (\tilde{u} - \tilde{v}^{\lambda}) &= \lim_{\varepsilon \rightarrow 0}\sup_{(t,x,m)\in\mathbf{Z} \times \mathbf{M}} \Gamma_{\varepsilon}(t, x, x, m)= \lim_{\varepsilon \rightarrow 0}\Gamma_{\varepsilon}(t_{\varepsilon}, x_{\varepsilon}, x_{\varepsilon}, m_{\varepsilon} )= \lim_{\varepsilon \rightarrow 0}\Gamma_{\varepsilon}(\hat t_{\varepsilon}, \hat x_{\varepsilon}, \hat x_{\varepsilon}, \hat m_{\varepsilon}),\label{comparaison_gamma_sup}  
	\end{align}
in which the last identity follows from \reff{comparaison_standard}. 
Combined with the above inequality, this shows that 
	$
		\sup_{\mathbf{Z} \times \mathbf{M}} (\tilde{u} - \tilde{v}^{\lambda})+ \lambda\delta \leq \lim_{\varepsilon \rightarrow 0} \mathcal{K}(\tilde{u} - \tilde{v}^{\lambda})(\hat{t}_{\varepsilon}, \hat{x}_{\varepsilon}, \hat{m}_{\varepsilon}) ,
	$
which leads to a contradiction for $\varepsilon$ small enough.

Step 4.  We now show that there is a subsequence such that $t_{n}^\eps < T$ for all $n \geq 1$. If not, one can assume that $t_{n}^\eps = T$ and it follows from the boundary condition (\ref{comparaison_bord}) and step 3 that $\tilde{u}(T, x_{n}^\eps, m_{n}^\eps) \leq \Kc_{T} \tilde{g}(T, x_{n}^\eps, m_{n}^\eps)$ for all $n \geq 1$. Since, by (\ref{comparaison_bord}) and Condition (iv) of Proposition \ref{comparaison}, $\tilde{v}^{\lambda}(T, y_{n}^\eps, m_{n}^\eps) \geq \Kc_{T}\tilde{g}(T, y_{n}^\eps, m_{n}^\eps)$, it follows that $\tilde{u}(T, x_{n}^\eps, m_{n}^\eps) - \tilde{v}^{\lambda}(T, y_{n}^\eps, m_{n}^\eps) \leq \Kc_{T}\tilde{g}(T, x_{n}^\eps, m_{n}^\eps) - \Kc_{T}\tilde{g}(T, y_{n}^\eps, m_{n}^\eps)$. Hence,
	$
		\Gamma_{\varepsilon}(T, x_{n}^\eps, y_{n}^\eps, m_{n}^\eps) \leq  \Kc_{T} \tilde{g}(T, x_{n}^\eps, m_{n}^\eps) -  \Kc_{T} \tilde{g}(T, y_{n}^\eps, m_{n}^\eps).
	$
Combining \reff{eq: hyp conti KcT et Kc}, 	(\ref{comparaison_standard}) and (\ref{comparaison_gamma_sup}) as above, we obtain 
$\sup (\tilde{u} - \tilde{v}^{\lambda})\le 0$, a contradiction. 
\medskip

Step 5.  In view of step 3 and 4, we may assume that
	$
		t_{n}^\eps < T $ and $ (\tilde{u} - \mathcal{K}\tilde{u})(t_{n}^\eps, x_{n}^\eps, m_{n}^\eps) > 0 $  for all  $ n \geq 1.
	$
Using Ishii's Lemma and following standard arguments, see Theorem 8.3 and the discussion after Theorem 3.2 in \cite{crandall1992user}, we deduce from the sub- and supersolution viscosity property of $\tilde{u}$ and  $\tilde{v}^{\lambda}$, and the Lipschitz continuity assumptions on $\mu$ and $\sigma$, that 
	\[
		\varrho\left(\tilde{u}(t_{n}^\eps, x_{n}^\eps, m_{n}^\eps)-\tilde{v}^{\lambda}(t_{n}^\eps, y_{n}^\eps, m_{n}^\eps)\right) \leq C\left( n \|x_{n}^\eps - y_{n}^\eps\|^{2} + \varepsilon\left( 1+ \|x_{n}^\eps\|^{2}+ \|y_{n}^\eps\|^{2}\right)\right),
	\]
	for some $C>0$, independent on $n$ and $\eps$. 
In view of (\ref{gamma_vareps}) and (\ref{comparaison_standard}), we get
	\begin{equation}\label{comparaison_fin}
		\varrho \Gamma_{\varepsilon}(\hat t_{\varepsilon},\hat  x_{\varepsilon}, \hat x_{\varepsilon}, \hat m_{\varepsilon})  \leq 2C \varepsilon\left(1 + \|\hat x_{\varepsilon}\|^{2}\right).
	\end{equation}
We shall prove in next step that the right-hand side of (\ref{comparaison_fin}) goes to 0 as $\varepsilon \rightarrow 0$, up to a subsequence. Combined with (\ref{comparaison_gamma_sup}), this  leads to a contradiction to (\ref{comparaison_0}).

\medskip

Step 6. We conclude the proof by proving the claim used above. First note that we can always construct a sequence $(\tilde t_{\eps},\tilde x_{\eps}, \tilde m_{\eps})_{\eps>0}$ such that 
$$
\Gamma_{\varepsilon}(\tilde t_{\eps},\tilde x_{\eps},\tilde x_{\eps}, \tilde m_{\eps}) \to \sup_{\Zb\x \Mb} (\tilde u-\tilde v^{\lambda})\;\;\mbox{ and } \;\;\eps( \|\tilde x_{\varepsilon}\|^{2}+{\rm d}_{\Mb}(\tilde m_{\eps}))\to 0 \;\;\mbox{ as } \eps\to 0.
$$
By  \reff{comparaison_standard}, 
$
\Gamma_{\varepsilon}(\tilde t_{\eps},\tilde x_{\eps},\tilde x_{\eps}, \tilde m_{\eps}) \le \Gamma_{\varepsilon}(\hat t_{\varepsilon},\hat  x_{\varepsilon}, \hat x_{\varepsilon}, \hat m_{\varepsilon}).
$
 Hence, 
$
 \sup_{\Zb\x \Mb} (\tilde u-\tilde v^{\lambda})\le \sup_{\Zb\x \Mb} (\tilde u-\tilde v^{\lambda})-2\liminf_{\eps\to 0} \varepsilon  \|\hat x_{\varepsilon}\|^{2}.
 $
  \ep
 
 
 \bibliographystyle{plain}

\end{document}